%% file: main.tex
\def\year{2021}\relax
\documentclass[letterpaper]{article} 
\usepackage{aaai21}  
\usepackage{times}  
\usepackage{helvet} 
\usepackage{courier}  
\usepackage[hyphens]{url}  
\usepackage{graphicx} 
\usepackage{natbib}  
\usepackage{caption} 
\usepackage[switch]{lineno}
\usepackage{amsfonts} 
\usepackage{amsmath}
\usepackage{amsthm}
\usepackage{subcaption}
\usepackage{algorithm}
\usepackage{algpseudocode}
\usepackage{graphicx}
\usepackage{tikz}
\usepackage{3dplot}

\usepackage[hyperindex, breaklinks, hypertexnames=false]{hyperref}

\usetikzlibrary{shadings,intersections,angles,quotes}

\newtheorem{theorem}{Theorem}
\newtheorem{lemma}{Lemma}

\input{commands}

\title{Smooth Convex Optimization using Sub-Zeroth-Order Oracles}
\author {
    Mustafa O. Karabag, 
    Cyrus Neary, 
    Ufuk Topcu \\ 
}
\affiliations {
    The University of Texas at Austin, Austin, TX \\
    \{karabag, cneary, utopcu\}@utexas.edu
}
\begin{document}
\maketitle

\begin{abstract}
We consider the problem of minimizing a smooth, Lipschitz, convex function over a compact, convex set using sub-zeroth-order oracles:  an oracle that outputs the sign of the directional derivative for a given point and a given direction, an oracle that compares the function values for a given pair of points, and an oracle that outputs a noisy function value for a given point. We show that the sample complexity of optimization using these oracles is polynomial in the relevant parameters. The optimization algorithm that we provide for the comparator oracle is the first algorithm with a known rate of convergence that is polynomial in the number of dimensions. We also give an algorithm for the noisy-value oracle that incurs a regret of $\tilde{\mathcal{O}}(\dimension^{3.75} \timehorizon^{0.75})$ (ignoring the other factors and logarithmic dependencies) where $\dimension$ is the number of dimensions and $\timehorizon$ is the number of queries.
\end{abstract}

\input{introduction.tex}
\input{prelims.tex}

\input{oracles.tex}
\input{regret.tex}

\input{conclusion.tex}

\clearpage

\section*{Ethics Statement}
The results of this paper are theoretical. The impacts of these results are thus dependent on the chosen application area. 

Learning human preferences is one of the possible applications of the optimization algorithms that we provide. While humans may not be able to express their utility functions or want to share their utility functions, they can rank the available prospects or evaluate a given prospect. The optimization algorithms provided in this paper can optimize a system based on human feedback. This application would positively affect the quality of human life. On the other hand, gradient-free optimization methods have practically high sample complexities. The use of these optimization methods may overwhelm humans due to the high number of queries and raise privacy concerns due to the high amount of collected data. We remark that while the sample complexities are high, the optimization algorithms that we provide do not need to store the query results. Hence, these algorithms are robust against data breaches.

\bibliography{ref2}	

\newpage

\onecolumn
\begin{center}
\huge \textbf{Smooth Convex Optimization using Sub-Zeroth-Order Oracles}\\
\LARGE Supplementary Material
\end{center}

We give the proofs for the technical results in this document. We also include the preliminaries for self-containment.

\input{supp_mat/supp_prelims}
\input{supp_mat/supp_proofs}

\end{document}

%% file: commands.tex
\newcommand{\absx}[1]{\left \lvert {#1} \right \rvert}

\newcommand{\halfspace}{\mathcal{H}}
\newcommand{\ellipsoid}{\mathcal{E}}
\newcommand{\ball}{\mathcal{B}}
\newcommand{\inscribedball}{\overline{\mathcal{B}}}
\newcommand{\circumscribedball}{\underline{\mathcal{B}}}

\newcommand{\circumscribedellipsoid}{\underline{\mathcal{E}}}
\newcommand{\cone}{\mathcal{F}}
\newcommand{\standarddev}{\sigma}
\newcommand{\smoothness}{\beta}
\newcommand{\lipschitz}{L}
\newcommand{\isotrans}[2]{T_{#1,#2}}
\newcommand{\isotransinverse}[2]{T^{-1}_{#1,#2}}
\newcommand{\eigenvalue}{\lambda}
\newcommand{\volumeofunitball}{\mathcal{V}_{\dimension}}
\newcommand{\timehorizon}{T}
\newcommand{\algo}{\mathcal{A}}
\newcommand{\regret}[1]{\mathcal{R}^{#1}}

\newcommand{\func}{f}
\newcommand{\funciso}{g}
\newcommand{\dimension}{n}
\newcommand{\Rdim}{\mathbb{R}^{n}}

\newcommand{\radiusofset}[1]{R_{#1}}
\newcommand{\suboptimality}{\varepsilon}
\newcommand{\failureprob}{\delta}
\newcommand{\convexset}{C}

\newcommand{\innerproduct}[2]{\left \langle {#1}, {#2} \right \rangle }

\newcommand{\norm}[1]{\left \Vert {#1} \right \Vert }

\newcommand{\curlyset}[1]{\left \lbrace {#1} \right \rbrace }
\newcommand{\ceilx}[1]{\left \lceil {#1} \right \rceil }
\newcommand{\oracle}[1]{\psi^{#1}}
\newcommand{\oraclewithoutname}{\psi}

\renewcommand{\arcsin}{\sin^{-1}}
\renewcommand{\arccos}{\cos^{-1}}
\newcommand{\identity}{I}

%% file: introduction.tex
\section{Introduction}

Derivative-free optimization methods are necessary when explicit access to the objective function is not available, or when the function's gradient is hard to compute ~\cite{conn2009introduction}. Utility functions, a concept from economics, provide an example of a type of objective function which may be hard to explicitly characterize. However, while a consumer may not be able to quantify their utility for a given prospect, they will likely be able to rank the available prospects. From a human's perspective, ranking the prospects may be simple, even if it is difficult to directly assign them values~\cite{abbas2015foundations}. For example, consider a reinforcement learning scenario in which a robot learns to perform a task via human feedback. The human may not be able to assign explicit rewards to the demonstrations performed by the robot, but she can rank them~\cite{akrour2012april,furnkranz2012preference,NIPS2012_4805}. While necessary in a range of applications~\cite{conn2009introduction,audet2017derivative}, derivative-free optimization methods are usually inferior in theory compared to first-order optimization methods~\cite{conn2009introduction}. In this work, we leverage the smoothness and convexity of the objective function to provide theoretical guarantees for derivative-free optimization.

We consider the problem of minimizing a smooth, Lipschitz continuous, convex function $\func$ on a convex, compact domain $\convexset \subset \Rdim $ using sub-zeroth-order oracles: i) the \textit{directional-preference oracle} that outputs the sign of the directional derivative for a given point and direction, ii) the \textit{comparator oracle} that compares the function value for two given points, and iii) the \textit{noisy-value oracle} that outputs the function value plus a subgaussian noise. 

For the directional-preference and comparator oracles, we prove an upper bound on the sample complexity that is polynomial in the relevant parameters. Our algorithms take advantage of the convexity and smoothness of the objective function, and rely on gradient estimation. We show that the direction of the gradient can be estimated with high accuracy via the sub-zeroth-order oracles. Having estimated the direction of the gradient, we use a variant of the ellipsoid method~\cite{shor1972utilization,judin1976evaluation}. We show that the sample complexity is $\tilde{\mathcal{O}}(\dimension^{4})$ for the directional-preference and comparator oracles. To the best of our knowledge, the optimization algorithm that we provide for the comparator oracle is the first algorithm with a known polynomial rate of convergence for smooth, convex functions.

We also develop a sublinear regret algorithm for the noisy-value oracle. The algorithm incurs $\tilde{\mathcal{O}}(\dimension^{3.75}\timehorizon^{0.75})$ regret (ignoring the other factors) with high probability where $\timehorizon$ is the number of queries. The best known high probability regret bound for the noisy-value oracle is $\tilde{\mathcal{O}}(\dimension^{9.5} \sqrt{\timehorizon})$~\cite{bubeck2017kernel}. While our algorithm requires smoothness, and its regret is not optimal in terms of the dependency on the number of queries, its lower order dependency on the number of dimensions makes it appealing compared to this existing regret bound. 

\paragraph{Related work}
The bisection method~\cite{burden19852} uses the directional-preference oracle to optimize a one-dimensional function. In multiple dimensions, Qian et al. (\citeyear{qian2015learning}) used the directional-preference oracle to optimize a linear function. Their algorithm uses a predefined set of query directions, whereas we consider a setting where the algorithm is allowed to query any direction at any point. SignSGD~\cite{bernstein2018signsgd} requires the sign of directional derivatives only for fixed orthogonal basis vectors and converges to the optimum for smooth, convex functions. SignSGD enjoys lower order dependency $\mathcal{O}(\dimension)$ on the number of dimensions. However, it has a sub-linear rate of convergence whereas our algorithm has a linear rate of convergence. Additionally, our algorithm for the directional-preference oracle also works for non-smooth functions.

Optimization using the comparator oracle was explored with directional direct search methods~\cite{audet2006mesh} and the Nelson-Mead method~\cite{nelder1965simplex}. Directional direct search is guaranteed to converge to an optimal solution in the limit for smooth, convex functions. However, the algorithm does not have a known rate of convergence. Meanwhile, the Nelson-Mead method may fail to converge to a stationary point for smooth, convex functions~\cite{mckinnon1998convergence}. Convergent variants of the Nelson-Mead method use function values in addition to comparator oracle queries~\cite{price2002convergent}. 

For the regret using the noisy-value oracle, a lower bound of $\Omega(\dimension \sqrt{\timehorizon})$ has been shown~\cite{shamir2013complexity}. Recently, Lattimore~(\citeyear{lattimore2020improved}) gave an existence result for an algorithm that achieves $\tilde{\mathcal{O}}(\dimension^{2.5} \sqrt{\timehorizon})$ regret in the adversarial case. The best known upper bounds with explicit algorithms are $\tilde{\mathcal{O}}(\dimension^{9.5} \sqrt{\timehorizon})$~\cite{bubeck2017kernel} and $\mathcal{O}(\dimension \timehorizon^{0.75})$~\cite{flaxman2004online} for Lipschitz, convex functions in the adversarial case. The regret bound $\mathcal{O}(\dimension^{3.75} \timehorizon^{0.75})$ that we provide is better than $\tilde{\mathcal{O}}(\dimension^{9.5} \sqrt{\timehorizon})$ regret bound of \cite{bubeck2017kernel} if $\timehorizon = o(\dimension^{23})$. Our result differs from~\cite{flaxman2004online} in that our algorithm succeeds with high probability whereas the algorithm given in \cite{flaxman2004online} succeeds in expectation.

We give the the proofs of the technical results in the supplementary material. 

%% file: prelims.tex
\section{Preliminaries} \label{section:prelims}
We denote the unit vectors in $\Rdim$ as $e_{1}, \ldots, e_{\dimension}$. Let $S$ be a set of vectors in $\Rdim$. $Proj_{S}(x)$ denotes the orthogonal projection of $x$ onto the span of $S$ and $Proj_{S^{\bot}}(x)$ denotes the orthogonal projection of $x$ onto the complement space of the span of $S$. The angle between $x$ and $y$ is $\angle(x,y)$. 

A convex function $\func: \convexset \to \mathbb{R}$ is said to be \textit{$\lipschitz$-Lipschitz} if $	\norm{\func(x) - \func(y)} \leq L \norm{x - y}$ for all $x,y \in \convexset$. A differentiable convex function $\func: \convexset \to \mathbb{R}$ is said to be \textit{$\smoothness$-strongly smooth} if $ |	\func(y) - \func(x) -  \innerproduct{\nabla \func(x)}{y-x}  | \leq  \smoothness \norm{y-x}^2 / 2 $ for all $x,y \in \convexset$.

The \textit{radius} $\radiusofset{\convexset}$ of a compact convex set $\convexset$ is equal to the the radius of the circumscribing ball, i.e., $\radiusofset{\convexset} = \min_{y \in \convexset} \max_{x \in \convexset} \norm{x-y}.$ A \textit{right circular cone} in $\Rdim$ with \textit{semi-vertical angle} $\theta \in [0, \pi/2]$ and \textit{direction} $v \in \Rdim$ is $
\cone(v, \theta) = \curlyset{w | W \in \Rdim, \angle(v,w) \leq \theta } $. An \textit{ellipsoid} in $\Rdim$ is $\ellipsoid(A,x_{0}) =  \curlyset{x | (x-x_{0})^{T} A^{-1} (x - x_{0}) \leq 1} $
where $x_{0} \in \Rdim$ and $A \in \mathbb{R}^{n \times n}$ is a positive definite matrix. The \textit{isotropic transformation} $\isotrans{A}{x_{0}}$ of an ellipsoid $\ellipsoid(A, x_{0})$ is $\isotrans{A}{x_{0}}(x) = A^{-1/2}(x - x_{0})\sqrt{\eigenvalue_{\max}(A)}.$ Intuitively, $\isotrans{A}{x_{0}}$ transforms the ellipsoid into a hypersphere centered at the origin with a radius equal to the ellipsoid's largest radius. The inverse of $\isotrans{A}{x_{0}}$ is $\isotrans{A}{x_{0}}^{-1}(x) = A^{1/2}x/\sqrt{\eigenvalue_{\max}(A)} + x_{0}.$  The \textit{circumscribing ellipsoid} $\circumscribedellipsoid_{\convexset} = \ellipsoid(A^{*}, x^{*}_{0})$ of a compact convex set $\convexset$ satisfies $\det(A^{*}) = \min_{A, x_{0}} \det(A)$ where $\convexset \subseteq \ellipsoid(A, x_{0}) $. We denote the \textit{identity matrix} by $\identity$.

A \textit{$\standarddev^2$-subgaussian} random variable with mean $\mu$ satisfies 
$\Pr(|X - \mu| > t) \leq 2\exp\left( -t^2/(2\standarddev^2) \right).$

%% file: oracles.tex
\section{Smooth convex optimization using sub-zeroth-order oracles} 
We consider the minimization of a $\smoothness$-smooth, $\lipschitz$-Lipschitz, convex function $\func$ on a compact, convex set $\convexset \subseteq \Rdim$ where $x^{*}$ denotes a minimizer of $\func$. We assume $x^{*}$ is an interior point of $\convexset$ such that $\ellipsoid( \suboptimality \identity / \lipschitz, x^{*}) \subseteq \convexset$, where $\suboptimality$ is the desired suboptimality gap. This assumption is included for simplicity, but can be removed by considering a near-optimal interior point with a sufficiently large neighborhood. Such a point is guaranteed to exist after the isotropic transformation. We also assume $\dimension \geq 2$, but the algorithms that we present generalize to the one-dimensional setting. 

\subsection{Sub-zeroth-order oracles} 
\label{section:problemoracles}

The first oracle we consider is the directional-preference oracle which outputs a binary value indicating whether the function is increasing on the queried direction at the queried point. The \textit{directional-preference oracle} $\oracle{DP}:\convexset \times \Rdim \to \lbrace -1 , 1  \rbrace$ is a function such that $\oracle{DP}(x,y) = -1$ if $\innerproduct{\nabla\func(x)}{y}  < 0$, and $\oracle{DP}(x,y) = 1$ otherwise.

We also consider the comparator oracle, which compares the function at a pair of query points.
The \textit{comparator oracle} $\oracle{\convexset}:\convexset \times \convexset \to \curlyset{-1, 1}$ is a function such that $\oracle{\convexset}(x,y) = -1$ if $\func(x)  \geq \func(y)$, and $\oracle{\convexset}(x,y) = 1$	otherwise. The comparator oracle is similar to the directional-preference oracle in that $\oracle{\convexset}(x, x+ky)$ approaches $\oracle{DP}(x,y)$ in the limit as $k$ approaches zero, i.e., $\lim_{k \to 0^{+}} \oracle{\convexset}(x, x+ ky) = \oracle{DP}(x,y)$ for all $x \in \convexset$ and $y \in \Rdim$.

The \textit{noisy value oracle} $\oracle{NV}: \convexset \to \mathbb{R}$ outputs the function value plus a $\sigma^2$-subgaussian noise, i.e., $\oracle{NV}(x) = f(x) + Z$ for all $x \in \convexset$, where $Z$ is a $\sigma^2$-subgaussian random variable with zero mean.

In addition to the sub-zeroth-order oracles, we also consider the zeroth-order value oracle as preliminary step for the noisy-value oracle. The \textit{value oracle} $\oracle{V}:\convexset \to \mathbb{R}$ outputs the function value at the queried point, i.e., $\oracle{V}(x) = \func(x)$ for all $x \in \convexset$.

\subsection{Ellipsoid method with approximate gradients} \label{section:optimize}

In this section, we provide optimization algorithms that employ the sub-zeroth-order oracles. We use a variation of the ellipsoid method~\cite{shor1972utilization,judin1976evaluation} that uses the approximately correct gradient direction. The ellipsoid method begins each iteration with an ellipsoid containing an optimal point, it then computes the function's gradient at the ellipsoid center and removes all points from the feasible set that lie along an ascent direction. The remaining points in the set are then enclosed in the minimum volume circumscribing ellipsoid, which is used as the starting ellipsoid in the next iteration. The volume of the generated ellipsoid decreases in each iteration. For a Lipschitz, convex function, this method is guaranteed to output a near optimal solution in a finite number of iterations. 

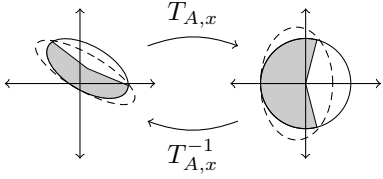
\begin{figure}[ht]
	
	\centering
		\begin{tikzpicture}[scale=1]

		\draw[line width=0.0mm, fill=gray!40]	(1.5,0) -- (1.65,0.5805) arc (75.52:284.47:0.6) -- cycle;
		\draw (1.5,0) ellipse (0.6 and 0.6);
		\draw[densely dashed] (1.38,0) ellipse (0.48 and 0.75);

		\draw[rotate around={60: (-1.4,0.2)} ] (-1.4,0.2) ellipse (0.3 and 0.6);
		\draw[rotate around={60: (-1.4,0.2)}, line width=0.0mm, fill=gray!40]	(-1.4,0.2) -- (-1.325,0.7805) arc[start angle=75.52,end angle=284.47, x radius=0.3, y radius=0.6] -- cycle;
		\draw[densely dashed, rotate around={60: (-1.4,0.2)} ] (-1.46,0.2) ellipse (0.24 and 0.75);
		
		\draw[<->] (-1.5, -1) -- (-1.5,1);
		\draw[<->] (-0.5, 0) -- (-2.5,0);
		
		\draw[<->] (1.5, -1) -- (1.5,1);
		\draw[<->] (0.5, 0) -- (2.5,0);	
		
		\path[->] (-0.6,0.5)  edge  [bend left=20]  node[above] {$\isotrans{A}{x}$} (0.6,0.5);
		\path[->] (0.6,-0.5)  edge  [bend left=20]  node[below] {$\isotransinverse{A}{x}$} (-0.6,-0.5);

		\end{tikzpicture}

	\caption{Illustrations of the ellipsoid cuts. The original coordinates are on the left and the isotropic coordinates on the right. The dashed ellipsoids enclose the shaded regions that are the possible descent directions. }
	\label{fig:transfromations}

\end{figure}

While the information on the gradient direction is sufficient to apply the classical ellipsoid method, computing the exact gradient direction would require infinitely many queries to the sub-zeroth-order oracles. On the other hand, if the semi-vertical angle of the cone of possible gradient directions is small enough, i.e., less than $\arcsin(1/\dimension)$, in the isotropic coordinates, one can still find an ellipsoid with a smaller volume that contains all possible descent directions and the optimal solution. 

\begin{lemma} \label{lemma:ellipsoid}
	Let $\func: \Rdim \to \mathbb{R}$ be a differentiable, convex function. For $\theta \in [0, \arcsin(1/\dimension)]$ and $p \in \Rdim$, if $\nabla \func(0) \in \cone(p, \theta)$, then $\func(x') \geq f(0)$ for all $x' \in \ellipsoid(I, 0) \cap 	 \curlyset{ x | \innerproduct{p / \norm{p}}{ x } > \sin\theta },$ and there exists an ellipsoid $\mathcal{E}^{*}$ such that $\mathcal{E}^{*} \supseteq \ellipsoid(I, 0) \cap 	 \curlyset{ x | \innerproduct{p / \norm{p}}{ x } \leq \sin\theta } $ and \[ \frac{Vol(\mathcal{E}^{*})}{Vol(\ellipsoid(I, 0)) } = \left( \frac{\dimension^2 (1- \sin^2(\theta))}{\dimension^2 - 1} \right)^{(\dimension-1)/2}  \frac{\dimension(1+\sin(\theta))}{\dimension+1} \]  If $\theta = \arcsin(1/(2\dimension))$, then \[Vol(\mathcal{E}^{*}) \leq Vol(\ellipsoid(I, 0)) e^{-\frac{1}{8(\dimension + 1)}} <Vol(\ellipsoid(I, 0)).\]
\end{lemma} 

Lemma \ref{lemma:ellipsoid} shows that if the semi-vertical angle is small enough, there exists an ellipsoid with a smaller volume that contains the intersection of the possible descent directions and the initial ellipsoid as shown in Figure \ref{fig:transfromations}. Since the isotropic transformation is affine, it preserves the ratio of volumes. Thus, there also exists an ellipsoid with a smaller volume in the original coordinates as shown in Figure \ref{fig:transfromations}.

We need to approximately estimate the direction of the gradient in order to employ Lemma \ref{lemma:ellipsoid}. For the value and the noisy-value oracles, we can estimate the direction of the gradient by sampling the function on a fixed set of basis vectors. However, to estimate the gradient direction using the comparator and directional-preference oracles, we need to successively select different collections of vectors along which to sample the function. In the following two sections, we describe in detail how to estimate the direction of the gradient using the sub-zeroth-order oracles, and how to use these estimations for optimization. 

\tdplotsetmaincoords{70}{110}


\pgfmathsetmacro{\pdoner}{1}
\pgfmathsetmacro{\pdonetheta}{0}
\pgfmathsetmacro{\pdonephi}{0}

\pgfmathsetmacro{\pdtwor}{1}
\pgfmathsetmacro{\pdtwotheta}{45}
\pgfmathsetmacro{\pdtwophi}{0}

\pgfmathsetmacro{\pdtwoantir}{1}
\pgfmathsetmacro{\pdtwoantitheta}{45}
\pgfmathsetmacro{\pdtwoantiphi}{180}

\pgfmathsetmacro{\pdthreer}{1}
\pgfmathsetmacro{\pdthreetheta}{45}
\pgfmathsetmacro{\pdthreephi}{90}

\pgfmathsetmacro{\pdthreeantir}{1}
\pgfmathsetmacro{\pdthreeantitheta}{45}
\pgfmathsetmacro{\pdthreeantiphi}{270}

\pgfmathsetmacro{\pdmidr}{1}
\pgfmathsetmacro{\pdmidtheta}{22.5}
\pgfmathsetmacro{\pdmidphi}{45}

\pgfmathsetmacro{\pdmidunr}{0.871}
\pgfmathsetmacro{\pdmiduntheta}{22.5}
\pgfmathsetmacro{\pdmidunphi}{45}
\begin{figure}[t]
	\centering
	\begin{subfigure}[b]{.12\textwidth}
		\centering
		\begin{tikzpicture}[scale=1.1,tdplot_main_coords]
		
		\coordinate (O) at (0,0,0);
		
		\tdplotsetcoord{pdone}{\pdoner}{\pdonetheta}{\pdonephi}
		\tdplotsetcoord{pdtwo}{\pdtwor}{\pdtwotheta}{\pdtwophi}
		\tdplotsetcoord{pdthree}{\pdthreer}{\pdthreetheta}{\pdthreephi}	
		\tdplotsetcoord{pdtwoanti}{\pdtwoantir}{\pdtwoantitheta}{\pdtwoantiphi}
		\tdplotsetcoord{pdthreeanti}{\pdthreeantir}{\pdthreeantitheta}{\pdthreeantiphi}	
		
		\draw[thick,->] (0,0,0) -- (1.2,0,0) node[anchor=north]{$d_{2}$};
		\draw[thick,->] (0,0,0) -- (0,1.2,0) node[anchor=north]{$d_{3}$};
		\draw[thick,->] (0,0,0) -- (0,0,1.2) node[anchor=south]{$d_{1}$};
		
		\draw[ color=black] (O) -- (pdone);
		\draw[ color=black] (O) -- (pdtwo) ;	
		\draw[ color=black] (O) -- (pdthree);
		\draw[dotted, color=black] (O) -- (pdtwoanti);	
		\draw[dotted, color=black] (O) -- (pdthreeanti);


		\tdplotsetrotatedcoords{0}{0}{0}	
		\tdplotdrawarc[tdplot_rotated_coords]{(0,0,0.707)}{0.707}{0}{360}{anchor=south west}{}
		\tdplotsetrotatedcoords{0}{270}{0}	
		\tdplotdrawarc[dotted, tdplot_rotated_coords]{(0,0,0)}{1}{-45}{0}{anchor=south west}{}
		\tdplotdrawarc[tdplot_rotated_coords]{(0,0,0)}{1}{0}{45}{anchor=south west}{}
		
		\tdplotsetrotatedcoords{270}{270}{0}	
		\tdplotdrawarc[dotted, tdplot_rotated_coords]{(0,0,0)}{1}{-45}{0}{anchor=south west}{}
		\tdplotdrawarc[tdplot_rotated_coords]{(0,0,0)}{1}{0}{45}{anchor=south west}{}
				\tdplotsetrotatedcoords{0}{270}{0}	
				\tdplotdrawarc[tdplot_rotated_coords,  color=gray]{(O)}{0.2}{45}{0}{anchor= west}{}
				
				\node[color=gray] (a) at (0,0.35,0.15) {$\gamma$};
		
		\end{tikzpicture}
		\caption{}
		\label{fig:cone1}
	\end{subfigure} \quad
	\begin{subfigure}[b]{.12\textwidth}
		\centering
		\begin{tikzpicture}[scale=1.1,tdplot_main_coords]
		
		\coordinate (O) at (0,0,0);
		
		
		\tdplotsetcoord{pdone}{\pdoner}{\pdonetheta}{\pdonephi}
		\tdplotsetcoord{pdtwo}{\pdtwor}{\pdtwotheta}{\pdtwophi}
		\tdplotsetcoord{pdthree}{\pdthreer}{\pdthreetheta}{\pdthreephi}	
		\tdplotsetcoord{pdtwoanti}{\pdtwoantir}{\pdtwoantitheta}{\pdtwoantiphi}
		\tdplotsetcoord{pdthreeanti}{\pdthreeantir}{\pdthreeantitheta}{\pdthreeantiphi}	
		
		\draw[ thick,->] (0,0,0) -- (1.2,0,0) node[anchor=north]{$d_{2}$};
		\draw[thick,->] (0,0,0) -- (0,1.2,0) node[anchor=north]{$d_{3}$};
		\draw[thick,->] (0,0,0) -- (0,0,1.2)  node[anchor=south]{$d_{1}$};
		
		\draw[ color=black] (O) -- (pdtwo) node[anchor=north east]{};	
		\draw[ color=black] (O) -- (pdthree) node[anchor=north west]{};


		\tdplotsetrotatedcoords{0}{0}{0}	
		\tdplotdrawarc[tdplot_rotated_coords]{(0,0,0.707)}{0.707}{0}{90}{anchor=south west}{}
		\tdplotsetrotatedcoords{0}{270}{0}	
		\tdplotdrawarc[tdplot_rotated_coords]{(0,0,0)}{1}{0}{45}{anchor=south west}{}
		
		\tdplotsetrotatedcoords{270}{270}{0}	
		\tdplotdrawarc[tdplot_rotated_coords]{(0,0,0)}{1}{0}{45}{anchor=south west}{}
				\tdplotsetrotatedcoords{0}{270}{0}	
				\tdplotdrawarc[tdplot_rotated_coords,  color=gray]{(O)}{0.2}{0}{45}{anchor=south west}{}
						\node[color=gray] (a) at (0,0.35,0.15) {$\gamma$};
		\end{tikzpicture}
		\caption{}
		\label{fig:cone2}
	\end{subfigure} \quad
	\begin{subfigure}[b]{.12\textwidth}
		\centering
		\begin{tikzpicture}[scale=1.1,tdplot_main_coords]
		
		\coordinate (O) at (0,0,0);
		
		
		\tdplotsetcoord{pdone}{\pdoner}{\pdonetheta}{\pdonephi}
		\tdplotsetcoord{pdtwo}{\pdtwor}{\pdtwotheta}{\pdtwophi}
		\tdplotsetcoord{pdthree}{\pdthreer}{\pdthreetheta}{\pdthreephi}	
		\tdplotsetcoord{pdtwoanti}{\pdtwoantir}{\pdtwoantitheta}{\pdtwoantiphi}
		\tdplotsetcoord{pdthreeanti}{\pdthreeantir}{\pdthreeantitheta}{\pdthreeantiphi}	
		
		\tdplotsetcoord{pdmid}{\pdmidr}{\pdmidtheta}{\pdmidphi}	
		\tdplotsetcoord{pdmidun}{\pdmidunr}{\pdmiduntheta}{\pdmidunphi}
		
		\draw[thick,->] (0,0,0) -- (1.2,0,0) node[anchor=north]{ $d_{2}$};
		\draw[thick,->] (0,0,0) -- (0,1.2,0) node[anchor=north ]{$d_{3}$};
		\draw[thick,->] (0,0,0) -- (0,0,1.2) node[anchor=south]{$d_{1}$};
		
		\draw[ color=black] (O) -- (pdtwo) node[anchor=north east]{} ;	
		\draw[ color=black] (O) -- (pdthree) node[anchor=north west]{};
		
		\draw[->, color=black] (O) -- (pdmid) node[anchor=north west]{$p$};

		
		\tdplotsetrotatedcoords{45}{22.5}{-20.94}
		\tdplotdrawarc[dashed, tdplot_rotated_coords]{(0,0,0.871)}{0.536}{0}{360}{anchor=south west}{}

		\tdplotsetrotatedcoords{0}{0}{0}	
		\tdplotdrawarc[tdplot_rotated_coords]{(0,0,0.707)}{0.707}{0}{90}{anchor=south west}{}
		\tdplotsetrotatedcoords{0}{270}{0}	
		\tdplotdrawarc[tdplot_rotated_coords]{(0,0,0)}{1}{0}{45}{anchor=south west}{}
		
		\tdplotsetrotatedcoords{270}{270}{0}	
		\tdplotdrawarc[tdplot_rotated_coords]{(0,0,0)}{1}{0}{45}{anchor=south west}{}
		
				\tdplotsetrotatedcoords{22.5}{20.94-90}{8.42}
		\tdplotdrawarc[tdplot_rotated_coords,  color=gray]{(0,0,0)}{0.2}{0}{32.84}{anchor=west}{}
								\node[color=gray] (a) at (0,0.35,0.15) {$\gamma'$};
		
		\end{tikzpicture}
		\caption{}
		\label{fig:cone3}
	\end{subfigure}
	\caption{Illustrations the gradient pruning method by directional-preferences. $(a)$ The cone $\cone(d_{1}, \gamma)$ is the possible gradient directions. $(b)$ The quarter cone is the possible gradient directions after the queries. $(c)$ The dashed cone $\cone(p, \gamma')$ overapproximates possible gradient directions. }
	\label{fig:smallergradientcones}
\end{figure}
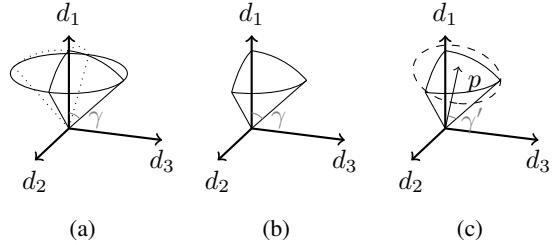

\subsubsection{Optimization using the directional-preference oracle}
For the directional preference oracle, we can estimate direction of the gradient by iteratively sampling the function along different sets of basis vectors. Consider Figure \ref{fig:smallergradientcones}  as an example. Assume that the gradient $\nabla \func(x)$ lies in  $\cone(d_{1}, \gamma)$ shown in Figure \ref{fig:cone1}. We can use $\oracle{DP}(x,d_{2})$ and $\oracle{DP}(x,d_{3})$ to prune the direction estimation. The query directions slice the $\dimension$-dimensional space into $2^{\dimension}$ hyperoctants that are symmetric around the direction of the cone. The query results determine the hyperoctant that the gradient lies in. For example, if $\oracle{DP}(x,d_{2}) = 1$ and $\oracle{DP}(x,d_{3}) = 1$, the gradient lies in the quarter cone given in Figure \ref{fig:cone2}. Before the next set of queries, we limit the possible set of gradient directions with  $\cone(p, \gamma')$ such that $\gamma' < \gamma$  as shown in Figure \ref{fig:cone3}.

\begin{lemma} \label{lemma:smallergradientcones}
	Let $\gamma \in (0, \pi/2],$ $d_{1} = e_{1},$ $d_{i} = \cos(\gamma)e_{1} + \sin(\gamma)e_{i},$ for all $ i\in \lbrace 2, \ldots, \dimension \rbrace$, $p = \sum_{i=1}^{\dimension} d_{i},$ and $\gamma' = \arccos(\langle p,d_{2} \rangle/ \|p\|) )$. 	Then, \(\mathcal{F} (p, \gamma' ) \supseteq \mathcal{F}(d_{1}, \gamma ) \cap \lbrace x | x_{i} \geq 0 \rbrace\) and \(\sin(\gamma')/\sin(\gamma) \leq \sqrt{\dimension-1} / \sqrt{\dimension}. \)
\end{lemma}

Lemma \ref{lemma:smallergradientcones} shows that if we choose the direction of the new cone as the average of the extreme points of the intersection of the previous cone and the hyperoctant as in Figure \ref{fig:cone3}, then the semi-vertical angle of the cone of possible gradient directions is a fraction of the previous angle depending on the number of dimensions. For the directional-preference and the comparator oracles, we repeat this process until the cone of possible gradient directions is sufficiently small, i.e., less than $\arcsin(1/(2\dimension))$.

Algorithm \ref{algo:optimizeDP} obtains a near-optimal solution for a given smooth, Lipschitz, convex function. At each iteration, we estimate the gradient direction using the direction pruning algorithm \textsc{PD-DP}, which implements the procedure described above. After the gradient direction estimation, we remove the ascent directions from the feasible set and proceed to the next iteration by enclosing the feasible set using an ellipsoid.

In the classical ellipsoid method, the output is the ellipsoid center with the smallest function value. The directional-preference oracle cannot compare the function values for a given pair of points, $x_{l}$ and $x_{r}$. However, we can use the bisection method to find a point $x'$ such that $\func(x') \leq \min(\func(x_{l}), \func(x_{r})) + \delta$ for a given $\delta$. Since the function is Lipschitz, the search stops after a finite number of iterations. To find a point whose function value is close to the function value of the optimal ellipsoid center,  we can remove $x^{l}$ and $x^{r}$ from the set of candidate points and add $x'$ to the set of the set of candidate points. Hence, the sample complexity of finding a point $x''$ such that $\func(x'') \leq \min_{x \in X} \func(x) + \suboptimality/2$ is linear in the size of $X$. The function \textsc{Compare-DP} implements the bisection search method on a given set $X$.

\begin{algorithm}[ht] 
	\caption{The optimization algorithm \textsc{Optimize-DP($X, \oracle{DP}$)} for the directional preference oracle } \label{algo:optimizeDP}
	\begin{algorithmic}[1]
		\State Find $\circumscribedellipsoid_{\convexset} = \ellipsoid(A^{(k)}, x^{(1)})$ of $\convexset$.
		\State Set $X = \lbrace x^{(1)} \rbrace$, $\convexset^{(1)} = \convexset$, $K = \left \lceil 8\dimension(\dimension+1) \log\left( \frac{2 \radiusofset{\convexset} \lipschitz}{\suboptimality}  \right) + 1  \right \rceil$.
		\For{$ k =1 \ldots K$}
		\State Set $p =$ \textsc{PD-DP}$\left( \oracle{DP}, x^{(k)}, \arcsin\left( 1/(2n)\right), A^{(k)} \right)$.
		\State Set $\convexset^{(k+1)} = \convexset^{(k)} \cap  \ellipsoid(A^{(k)}, x^{(k)}) \cap \isotransinverse{A^{(k)}}{x^{(k)}} \left(\curlyset{ x | \innerproduct{p / \norm{p}}{ x } \leq 1/(2\dimension) }\right)$.
		\State Find $\circumscribedellipsoid_{\convexset^{(k+1)}} = \ellipsoid(A^{(k+1)}, x^{(k+1)})$ of $\convexset^{(k+1)}$.
		\State Set $X = X \cup \lbrace x^{(k+1)} \rbrace$.
		\EndFor
		\State \textbf{return} $\textsc{Compare-DP}(X,\oracle{DP}, \suboptimality/2)$.
		
	\end{algorithmic}
	
\end{algorithm}

\begin{algorithm}[ht] 
	\caption*{\textbf{Function} \textsc{PD-DP}($ x, \theta, \isotrans{A}{x} $)} 
	\begin{algorithmic}[1]
	\State $p = e_{1}$, $r = 1$, $\gamma = \pi/2$.
	\While{$\gamma > \theta$}
	\State Find $d_{i}$ such that $d_{1} = p$, $d_{i} \perp d_{j}$ for all $i\neq j \in [n]$, and $\| d_{i} \| = 1$ for all $i \in [n]$. 
	\State Query $\oracle{DP}(x, A^{-1/2}d_{1}), \ldots, \oracle{DP}(x,A^{-1/2}d_{n})$. \label{algo:findperpdirections}
	\State Set $w_{1} = d_{1}$ and for all $i \in \lbrace 2, \ldots, n \rbrace$, set $w_{i} =  d_{1} \cos(\gamma) +  d_{i} \oracle{DP}(x, A^{-1/2}d_{i}) \sin(\gamma)$. \label{algo:projection}
	\State Set $p = \left( \sum_{i=1}^{n} w_{i}/n \right)/ \left\| \sum_{i=1}^{n} w_{i}/n  \right \|$, \label{algo:findthesmallercone}
	\State Set $\gamma = \arccos(\langle p, w_{2} \rangle)$.
	\State Set $r = \arcsin (\gamma)$. \label{algo:assignment}
	\EndWhile
	\State \textbf{return} $p$.
	\end{algorithmic}
\end{algorithm}

\begin{algorithm}[ht] 
	\caption*{\textbf{Function} \textsc{Compare-DP($X, \suboptimality$)}  } \label{algo:DPcomparison}
	\begin{algorithmic}[1]
		\State Set $X^{*} = X$ and $m = |X|$.
		\While{$|X^{*}| > 1$}
		\State Arbitraritly pick $x^{1}, x^{2} \in X$ such that $x^{1} \neq x^{2}$.
		\State Set $X^{*} = X^{*} \setminus \lbrace x^{1}, x^{2} \rbrace$.
		\State Set $x^{l} = x^{1}$ and $x^{r} = x^{2}$.
		\While{$\| x^{r} - x^{l}\| \leq  2\suboptimality/(\lipschitz m)$}
		\State Query $\oracle{DP}((x^{r} + x^{l})/2, (x^{r} - x^{l})/2)$.
		\If{$\oracle{DP}((x^{r} + x^{l})/2, (x^{r} - x^{l})/2) = 0$}
		\State $x^{l} = (x^{r} + x^{l})/2.$
		\Else
		\State $x^{r} = (x^{r} + x^{l})/2.$
		\EndIf
		\EndWhile
		\State $X^{*} = X^{*} \cup \lbrace (x^{r} + x^{l})/2 \rbrace$
		\EndWhile
		\State \textbf{return} $x^{*} \in X^{*}$.	
	\end{algorithmic}
	
\end{algorithm}

\begin{theorem} \label{theorem:DP} Let $K = \ceilx{8 \dimension(\dimension +1) \log\left( \frac{2 \radiusofset{\convexset \lipschitz}}{\suboptimality} \right)}.$ For an $\lipschitz$-Lipschitz, $\smoothness$-smooth, convex function $\func: \convexset \to \mathbb{R}$, Algorithm \ref{algo:optimizeDP} makes at most
\[\dimension K \ceilx{2\dimension \log(2\dimension)}  + K \log_{2}\left( \frac{\radiusofset{\convexset} \lipschitz (K + 1) }{\suboptimality} \right) \] queries to $\oracle{DP}$ and the output $x'$ of Algorithm \ref{algo:optimizeDP} satisfies $\func(x') \leq \min_{x \in \convexset} \func(x) + \suboptimality$.
\end{theorem}

The sample complexity and the correctness of Algorithm \ref{algo:optimizeDP} follows from Lemmas \ref{lemma:ellipsoid} and \ref{lemma:smallergradientcones}. The sample complexity using the directional-preference oracle is $\tilde{\mathcal{O}}(\dimension^2)$ of the classical ellipsoid algorithm. An invetable factor of $\mathcal{O}(\dimension)$ is required to query the function in all dimensions, i.e., to slice the cone of the possible gradient directions into hyperoctants. By Lemma \ref{lemma:smallergradientcones}, a factor of $\mathcal{O}(\dimension \log(\dimension))$ is due to the number of iterations of the gradient pruning algorithm. While the gradient pruning method is optimal when the semi-vertical angle of the possible gradient directions is large, it is suboptimal when the semi-vertical angle is close to $0$. One may improve the dependency of $\mathcal{O}(\dimension \log(\dimension))$ by treating this small angle regime differently.  We remark that optimization using the directional-preference oracle is still possible in the absence of smoothness. One can use the same optimization method with an oracle that outputs the sign of an arbitrary directional subgradient.

\subsubsection{Optimization using the comparator oracle}
The optimization algorithm that we provide for the comparator oracle is similar to the optimization algorithm for the directional preference oracle. To solve the optimization problem, we begin by using comparisons to infer the sign of the directional derivative,  i.e., we use the comparator oracle $\oracle{C}$ to infer the directional-preference oracle $\oracle{DP}$. Then, we approximately find the direction of the gradient using the signs of the directional derivatives.

\begin{figure}[b]
	\centering
	\begin{subfigure}[b]{.25\linewidth}
		\centering
		\begin{tikzpicture}[scale=0.6]    
		\draw[<->] (-2,0) -- (2,0);
		\draw [black] plot [smooth] coordinates {(-1.5,2) (-1,1.5) (0,1) (1,0.7) (1.5,0.6)};
		\filldraw[black] (-1,0) circle (1.5pt) node[anchor=north] {$x_{l}$};
		\filldraw[black] (0,0) circle (1.5pt) node[anchor=north] {$x_{m}$};
		\filldraw[black] (1,0) circle (1.5pt) node[anchor=north] {$x_{r}$};
		\draw[dotted] (-1.5, 1.5) -- (1.5, 1.5); 
		\draw[dotted] (-1.5, 1) -- (1.5, 1); 
		\draw[dotted] (-1.5, 0.7) -- (1.5, 0.7); 
		\end{tikzpicture}
		\caption{}
		\label{fig:threecases:case1}
	\end{subfigure}%
	\quad \
	\begin{subfigure}[b]{.25\linewidth}
		\centering\begin{tikzpicture}    [scale=0.6] 
		\draw[<->] (-2,0) -- (2,0);
		\draw [black] plot [smooth] coordinates {(-1.5,0.6) (-1,0.7) (0,1) (1,1.5) (1.5,2)};
		\filldraw[black] (-1,0) circle (1.5pt) node[anchor=north] {$x_{l}$};
		\filldraw[black] (0,0) circle (1.5pt) node[anchor=north] {$x_{m}$};
		\filldraw[black] (1,0) circle (1.5pt) node[anchor=north] {$x_{r}$};
		\draw[dotted] (-1.5, 1.5) -- (1.5, 1.5); 
		\draw[dotted] (-1.5, 1) -- (1.5, 1); 
		\draw[dotted] (-1.5, 0.7) -- (1.5, 0.7); 
		\end{tikzpicture}
		\caption{}
		\label{fig:threecases:case2}
	\end{subfigure}
	\quad \
	\begin{subfigure}[b]{.25\linewidth}
		\centering
		\centering\begin{tikzpicture} [scale=0.6]    
		\draw[<->] (-2,0) -- (2,0);
		\draw (-1.5,2) parabola bend (0,0.7) (1.5,1.2);
		\filldraw[black] (-1,0) circle (1.5pt) node[anchor=north] {$x_{l}$};
		\filldraw[black] (0,0) circle (1.5pt) node[anchor=north] {$x_{m}$};
		\filldraw[black] (1,0) circle (1.5pt) node[anchor=north] {$x_{r}$};
		\draw[dotted] (-1.5, 1.2) -- (1.5, 1.2); 
		\draw[dotted] (-1.5, 0.9) -- (1.5, 0.9); 
		\draw[dotted] (-1.5, 0.7) -- (1.5, 0.7); 
		\end{tikzpicture}
		\caption{}
		\label{fig:threecases:case3}
	\end{subfigure}
	\caption{Possible orderings for a convex function at three points on a line. }
	\label{fig:threecases}
\end{figure}
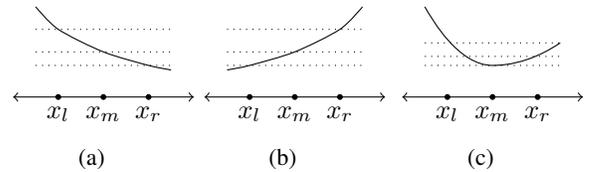

 Suppose function $g$ is in isotropic coordinates and we compare the function values at three points on a line, $x_{r}$, $x_{m}$, and $x_{l}$. We can get the directional derivative information at the middle point if the values of the function at  $x_{r}$, $x_{m}$, and $x_{l}$ are ordered as in Figures \ref{fig:threecases:case1} and \ref{fig:threecases:case2}. If the queried points are not ordered, i.e., the function value at $x_{m}$ is lower than or equal to the function values at both $x_{r}$ and $x_{l}$ as in Figure \ref{fig:threecases:case3}, the sign of the directional derivative is unknown at $x_{m}$. Function \textsc{FDD-C} takes the isotropic transformation information and outputs directional derivative information.

\begin{algorithm}[ht] 
	\caption*{\textbf{Function} \textsc{FDD-C}{($A, x_{0},  d, t$)}}
	\begin{algorithmic}[1]
		\State Query $\oracle{\convexset}(x - tA^{-1/2}d, x), \oracle{\convexset}(x, x+ tA^{-1/2}d)$.
		\State \textbf{if} {$\func(x - tA^{-1/2}d) \leq \func(x) \wedge \func(x ) \leq \func(x+ tA^{-1/2}d)$} \textbf{then}  \textbf{return} $1$.
		\State \textbf{else if} {$\func(x - tA^{-1/2}d) < \func(x) \wedge \func(x ) < \func(x+ tA^{-1/2}d)$} \textbf{then} \textbf{return} $-1$. 
		\State \textbf{else} \textbf{return} $unknown$. 
	\end{algorithmic}
\end{algorithm}

In cases when the sign of the directional derivative is unknown, we can use the smoothness of the objective function to bound the magnitude of the derivative as follows. In the case shown in Figure \ref{fig:threecases:case3}, there exists a point $x'$ such that $\innerproduct{\nabla \funciso(x')}{x^{r} - x^{l}} = 0$ and $x' = \alpha x^{r} + (1-\alpha)x^{l}$ for some $\alpha \in [0,1]$. Due to the smoothness property, we have $\innerproduct{\nabla \funciso(x^{m})}{x^{r} - x^{l}} \leq \beta \norm{x^{r} - x^{l}}/2$. 

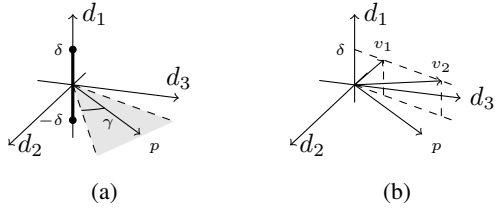
\begin{figure}[t]
	\centering
	\begin{subfigure}[b]{.4\linewidth}
		\centering
		\begin{tikzpicture}[scale=2.5,tdplot_main_coords]
		
		\coordinate (O) at (0,0,0);
		
		
		\tdplotsetcoord{pdone}{\pdoner}{\pdonetheta}{\pdonephi}
		\tdplotsetcoord{pdtwo}{\pdtwor}{\pdtwotheta}{\pdtwophi}
		\tdplotsetcoord{pdthree}{\pdthreer}{\pdthreetheta}{\pdthreephi}	
		\tdplotsetcoord{pdtwoanti}{\pdtwoantir}{\pdtwoantitheta}{\pdtwoantiphi}
		\tdplotsetcoord{pdthreeanti}{\pdthreeantir}{\pdthreeantitheta}{\pdthreeantiphi}	
		
		\tdplotsetcoord{pdmid}{\pdmidr}{\pdmidtheta}{\pdmidphi}	
		\tdplotsetcoord{pdmidun}{\pdmidunr}{\pdmiduntheta}{\pdmidunphi}
		
		\draw[->] (-0.2,0,0) -- (1,0,0) node[anchor=west]{$d_{2}$};
		\draw[->] (0,-0.2,0) -- (0,0.6,0) node[anchor=south]{$d_{3}$};
		\draw[->] (0,0,-0.3) -- (0,0,0.4) node[anchor=west]{$d_{1}$};
		
		\fill[fill=gray!20] (0,0,0)--(1,0.5,0)--(0.35,0.7,0);
		\draw[dashed] (0,0,0) -- (1,0.5,0);
		\draw[->] (0,0,0) -- (0.6,0.6,0) node[anchor=north west] {\tiny $p$};
		\draw[dashed] (0,0,0) -- (0.35,0.7,0);

		\tdplotdrawarc{(O)}{0.4}{26.5651}{45}{anchor=north west}{\tiny $\gamma$}	
		\draw[very thick] (0,0,-0.2) -- (0,0,0.2);
		\filldraw[black] (0,0,-0.2)  circle (0.5pt) node[anchor=east] {\tiny $-\delta$};
		\filldraw[black] (0,0,0.2)  circle (0.5pt) node[anchor=east] {\tiny  $ \delta$};
		
		\end{tikzpicture}
		\caption{}
		\label{fig:twocases:uncertainitysets}
	\end{subfigure}
	\quad \
	\begin{subfigure}[b]{.4\linewidth}
		\centering
		\begin{tikzpicture}[scale=2.5,tdplot_main_coords]
		
		\coordinate (O) at (0,0,0);
		
		
		\tdplotsetcoord{pdone}{\pdoner}{\pdonetheta}{\pdonephi}
		\tdplotsetcoord{pdtwo}{\pdtwor}{\pdtwotheta}{\pdtwophi}
		\tdplotsetcoord{pdthree}{\pdthreer}{\pdthreetheta}{\pdthreephi}	
		\tdplotsetcoord{pdtwoanti}{\pdtwoantir}{\pdtwoantitheta}{\pdtwoantiphi}
		\tdplotsetcoord{pdthreeanti}{\pdthreeantir}{\pdthreeantitheta}{\pdthreeantiphi}	
		
		\tdplotsetcoord{pdmid}{\pdmidr}{\pdmidtheta}{\pdmidphi}	
		\tdplotsetcoord{pdmidun}{\pdmidunr}{\pdmiduntheta}{\pdmidunphi}
		
		\draw[->] (-0.2,0,0) -- (1,0,0) node[anchor=west]{$d_{2}$};
		\draw[->] (0,-0.2,0) -- (0,0.6,0) node[anchor=west]{$d_{3}$};
		\draw[->] (0,0,-0.1) -- (0,0,0.4) node[anchor=west]{$d_{1}$};
		
		\draw[->] (0,0,0) -- (0.6,0.6,0) node[anchor=north west] {\tiny $p$};
		\draw[dashed] (0,0,0) -- (0.35,0.7,0);

		\filldraw[black] (0,0,0.2)  circle (0pt) node[anchor=east] {\tiny $\delta$};
		
		\draw[dashed] (0,0,0.2) -- (0.35,0.7,0.2);
		\draw[dashed] (0.1,0.2,0) -- (0.1,0.2,0.2);
		\draw[dashed] (0.3,0.6,0) -- (0.3,0.6,0.2);
		\draw[->] (0,0,0) -- (0.1,0.2,0.2) node[anchor=south] {\tiny $v_{1}$};
		\draw[->] (0,0,0) -- (0.3,0.6,0.2) node[anchor=south] {\tiny $v_{2}$};
		\end{tikzpicture}
		\caption{}
		\label{fig:twocases:uncertainitysets2}
	\end{subfigure}
	
	\caption{Possible cases for Algorithm \ref{algo:optimizeC}. $(a)$ The uncertainty sets for the unknown direction, $d_{1}$, and the known directions, $d_{2}$ and $d_{3}$. $(b)$ Two possible cases for the gradient estimation in Algorithm \ref{algo:optimizeC}.}
	\label{fig:twocases}
\end{figure}

\begin{algorithm}[ht] 
	\caption*{\textbf{Function}  \textsc{PD-C}{$(x, \theta, A, t )$}}
	\begin{algorithmic}[1]
		\State Set $r = 1$, $\gamma = \pi/2$, $m=0$, $UD = \emptyset$, $p = e_{1}$.
		\While{$\gamma > \theta \wedge m < \dimension$}
		\State Set $\curlyset{d_{1}, \ldots, d_{m}} = UD$.
		\State Find $d_{i}$ such that $d_{m+1} = p$, $d_{i} \perp d_{j}$ for all $i\neq j \in [\dimension]$, and $\| d_{i} \| = 1$ for all $i \in [\dimension]$.
		\State  Set $\oracle{DP}(x, A^{-1/2}d_{i}) =$ \textsc{FDD-C}$(A, x_{0},  d, t)$ for all $i \in [\dimension]$.
		\If{$\exists i \in \curlyset{m+1, \ldots, \dimension}$, such that $\oracle{DP}(x, A^{-1/2}d_{i}) =unknown$}
		\State  Set $UD = UD \cup d_{i}$,  and $m= m+1$.
		\Else  
		\State Set $w_{i} =  d_{m+1} \oracle{DP}(x, A^{-1/2}d_{m+1}) \cos(\gamma) +  d_{i} \oracle{DP}(x, A^{-1/2}d_{i}) \sin(\gamma)$ for all $i \in [n]$.
		\State Set $p = \left( \sum_{i=m+1}^{\dimension} w_{i}/\dimension \right)/ \left\| \sum_{i=m+1}^{\dimension} w_{i}/\dimension  \right \|$, $\gamma = \cos^{-1}(\langle p, w_{m+2} \rangle)$, $r = \arcsin (\gamma)$. 
		\EndIf
		\EndWhile
		\State \textbf{if} $m\neq \dimension$ \textbf{then} \textbf{return} $p$, \textbf{else} \textbf{return} $e_{1}$.
	\end{algorithmic}
\end{algorithm}

The function \textsc{PD-C} prunes the cone of the possible gradient directions by inferring the directional derivative information on different sets of basis vectors. At each iteration, the algorithm starts with a cone of possible gradient directions. Based on the query results the algorithm identifies the unknown directions $UD$ and finds an approximate direction for the projection $Proj_{UD^{\bot}}(\nabla \funciso(x))$ of the gradient onto the span of the known directions. In the next iteration, the algorithm uses the $Proj_{UD^{\bot}}(\nabla \funciso(x))$ as the direction of the cone of the possible gradient directions. When the semi-vertical angle of the cone of the possible gradient directions is sufficiently small or the number of unknown directions is equal to the number of dimensions, the function returns the estimation for the direction of the gradient.

\begin{algorithm}[ht] 
	\caption{The optimization algorithm \textsc{Optimize-C($\suboptimality$)} for the comparator oracle } \label{algo:optimizeC}
	\begin{algorithmic}[1]
		\State Set  $\convexset^{(1)} = \convexset$. Find $\circumscribedellipsoid_{\convexset^{(1)}} = \ellipsoid(A^{(1)}, x^{(1)})$ of $\convexset^{(1)}$.
		\State Set  $X = \lbrace x^{(1)} \rbrace$, $K = \left \lceil 8\dimension(\dimension+1) \log\left( \frac{ \radiusofset{\convexset} \lipschitz}{\suboptimality}  \right)  \right \rceil$, $\kappa = \max\left( \frac{4}{4 \dimension - \sqrt{2}\dimension\sqrt{\frac{4\dimension^2 - 1}{4\dimension^2}}} , 1 \right).$
		\For{$ k =1 \ldots K$}
		\State Set $t^{(k)} = \frac{\min(\suboptimality, \sqrt{\eigenvalue_{\max}(A^{k})} )}{\kappa \dimension^{5/2} \max(\smoothness, 1) \max(\radiusofset{\convexset}, 1)}.$ 
		\State Set $p$ = \textsc{PD-C}$\left(  x^{(k)}, \arcsin\left(\frac{1}{2\sqrt{2}\dimension}\right), A^{(k)}, t^{(k)} \right).$
		\State Set $\convexset^{(k+1)} = \convexset^{(k)} \cap  \ellipsoid(A^{(k)}, x^{(k)}) \cap \isotransinverse{A^{(k)}}{x^{(k)}} \left(\curlyset{ x | \innerproduct{p / \norm{p}}{ x } \leq 1/(2\dimension) }\right)$.
		\State Find $\circumscribedellipsoid_{\convexset^{(k+1)}} = \ellipsoid(A^{(k+1)}, x^{(k+1)})$ of $\convexset^{(k+1)}$. 
		\State Set  $X = X \cup \lbrace x^{(k+1)} \rbrace$.
		\EndFor
		\State Find $x' = \min_{x \in X} \func(x)$ using $\oracle{C}$.
		\State \textbf{return} $x'$.
	\end{algorithmic}
\end{algorithm}

Algorithm \ref{algo:optimizeC}, used for optimization with the comparator oracle, has two steps in each iteration. In the first step, the algorithm identifies a candidate approximate gradient direction in the isotropic coordinates using the direction pruning function $\textsc{PD-C}$. In the second step, the algorithm performs a cut as in the classical ellipsoid method. 

In order to find a near-optimal point, the algorithm exploits the fact that the direction of the projection of the gradient onto the linear subspace $Span(UD)^{\bot}$ of $\Rdim$ is approximately correct, and the magnitude of the projection $\norm{Proj_{UD}(\nabla \funciso(x))} $ of the gradient onto the complement subspace is small. For example, in Figure 	\ref{fig:twocases:uncertainitysets}, direction $d_{1}$ is the unknown direction and $\norm{Proj_{UD}(\nabla \funciso(x))} \leq \delta$. Directions $d_{2}$ and $d_{3}$ are the known directions and $\angle(Proj_{UD}(\nabla \funciso(x)), p) \leq \gamma$. There are two possible cases:
\begin{enumerate}
    \item The angle between $Proj_{UD^{\bot}}(\nabla \funciso(x))$ and $\nabla \funciso(x)$ is sufficiently small. \label{cases:correctangle}
    \item The angle between $Proj_{UD^{\bot}}(\nabla \funciso(x))$ and $\nabla \funciso(x)$ is not sufficiently small. \label{cases:smallgradient}
\end{enumerate}
Case \ref{cases:correctangle} happens if $\norm{Proj_{UD^{\bot}}(\nabla \funciso(x))}$ is large enough. In this case, the estimation for the direction of the gradient $\nabla \funciso(x)$ is approximately correct since the estimation $p$ for the direction of $Proj_{UD^{\bot}}(\nabla \funciso(x))$ is approximately correct. In this case, the ellipsoid algorithm proceeds normally. For example, if $\nabla \funciso(x) = v_{2}$ in Figure \ref{fig:twocases:uncertainitysets2}, then $\angle(\nabla \funciso(x), p)$ is small enough, say less than $\arcsin(1/(2\dimension))$. If Case \ref{cases:smallgradient} happens, the gradient approximation is not accurate, i.e., $\angle(\nabla \funciso(x), p)$ might be larger than $\arcsin(1/(2\dimension))$. However, if  Case \ref{cases:smallgradient} happens, it implies that $\norm{Proj_{UD^{\bot}}(\nabla \funciso(x))}$ is not large enough compared to $\norm{Proj_{UD}(\nabla \funciso(x))}$. Consequently, the magnitude $\norm{\nabla \funciso(x)}$ of the gradient is not large, say less than $\suboptimality/(\dimension \radiusofset{\convexset})$. For example, if $\nabla \funciso(x) = v_{1}$ in Figure \ref{fig:twocases:uncertainitysets2}, then $\norm{\nabla \funciso(x)}$ is small enough. We carefully choose the sampling distance so that the current ellipsoid center $x$ is near optimal if $\norm{Proj_{UD^{\bot}}(\nabla \funciso(x))}$ is not large enough. Algorithm \ref{algo:optimizeC} is agnostic to whichever case happens: The algorithm always assumes that the direction estimation approximately correct. However, the output point is near optimal since we compare the ellipsoid centers and output the best point before the termination.

\begin{theorem} \label{theorem:C} Let $K =  \left \lceil 8 \dimension (\dimension + 1) \log\left( \frac{ \radiusofset{\convexset} \lipschitz}{\suboptimality}  \right)  \right \rceil.$
For an $\lipschitz$-Lipschitz, $\smoothness$-smooth, convex function $\func: \convexset \to \mathbb{R}$, Algorithm \ref{algo:optimizeC} makes at most \[  2\dimension\left \lceil 2n\log(2\sqrt{2}\dimension) + \dimension \right \rceil K +  K
	\] queries to $\oracle{C}$ and the output $x'$ of Algorithm \ref{algo:optimizeC} satisfies $\func(x') \leq \min_{x \in \convexset} \func(x) + \suboptimality$.
\end{theorem}

Theorem \ref{theorem:C} shows that using the comparator oracle we can find a near optimal point with $\tilde{\mathcal{O}}(\dimension^{4})$ queries, which is at the same order with the sample complexity of optimization using the directional preference oracle. We also remark that while the smoothness of the function is required to determine the sampling distance, the sample complexity is not dependent on the smoothness constant.

\subsubsection{Optimization using the value oracle}

The value oracle is more informative than the comparator and directional-preference oracles; we can query the function in orthogonal directions near the center point and estimate the gradient. In the limit, i.e., the sampling distance goes to $0$, the gradient estimate converges to the true gradient. 

Under the smoothness assumption, we can get a provably good approximation of the gradient with a finite sampling distance. Let $\funciso$ be a $\smoothness$-smooth function in the isotropic coordinates. Formally, we have $\funciso(x) - \funciso(y) - \smoothness \norm{x-y}^2/2 \leq \innerproduct{\nabla \funciso(x)}{x-y} \leq \funciso(x) - \funciso(y) + \smoothness \norm{x-y}^2/2.$

\begin{figure}[ht]
	\centering
	\begin{subfigure}{.12\textwidth}
		\centering
		\begin{tikzpicture} 
		\draw [fill=gray!20, gray!20] (-0.2,1.0) rectangle (2,1.4);
		\draw [fill=gray!20, gray!20] (0.7,-0.2) rectangle (1.1,2);
		\draw [fill=gray!70, gray!70] (0.7,1.0) rectangle (1.1,1.4);
		\filldraw[black] (0.3,0.4) circle (0.1pt) node[anchor=west] {$\hat{\nabla \func(x)}$};
		\draw[->] (0,0) -- (0.9,1.2);
		\draw (0.9,1.2) circle (0.283cm);
		\draw[->] (0,-0.2) -- (0,2)   node[anchor=west] {$d_{2}$};  
		\draw[->] (-0.2,0) -- (2,0)  node[anchor=south] {$d_{1}$};
		\end{tikzpicture}
		\caption{}
		\label{fig:hypercubeestimate}
	\end{subfigure}%
	\quad \quad
	\begin{subfigure}{.12\textwidth}
		\centering
		\centering	\begin{tikzpicture} 
		\coordinate (grad) at (0.9,1.2);
		\coordinate (max) at (1.405,1.275);
		\coordinate (max2) at (0.814,1.651);
		\coordinate (origin) at (0,0);
		\draw [fill=gray!20, gray!20] (-0.2,1.0) rectangle (2,1.4);
		\draw [fill=gray!20, gray!20] (0.7,-0.2) rectangle (1.1,2);
		\draw [fill=gray!70, gray!70] (0.7,1.0) rectangle (1.1,1.4);
		\draw[->] (0,0) -- (grad);
		\draw (grad) circle (0.282cm);
		\draw[dashed] (0,0) -- (max);
		\draw[dashed] (0,0) -- (max2);
		\draw[->] (0,-0.2) -- (0,2)   node[anchor=west] {$d_{2}$};  
		\draw[->] (-0.2,0) -- (2,0)  node[anchor=south] {$d_{1}$};
		\end{tikzpicture}
		\caption{}
		\label{fig:largegradient}
	\end{subfigure}
	\quad \quad
	\begin{subfigure}{.12\textwidth}
		\centering				\begin{tikzpicture} 
		\coordinate (grad) at (0.4,0.4);
		\coordinate (max) at (0.852,0.228);
		\coordinate (max2) at (0.212,0.828);
		\coordinate (origin) at (0,0);
		\draw [fill=gray!20, gray!20] (-0.2,0.2) rectangle (2,0.6);
		\draw [fill=gray!20, gray!20] (0.2,-0.2) rectangle (0.6,2);
		\draw [fill=gray!70, gray!70] (0.2,0.2) rectangle (0.6,0.6);
		\draw[->] (0,0) -- (grad);
		\draw (grad) circle (0.282cm);
		\draw[dashed] (0,0) -- (max);
		\draw[dashed] (0,0) -- (max2);
		\draw[->] (0,-0.2) -- (0,2)   node[anchor=west] {$d_{2}$};  
		\draw[->] (-0.2,0) -- (2,0)  node[anchor=south] {$d_{1}$};
		\end{tikzpicture}
		\caption{}
		\label{fig:smallgradient}
	\end{subfigure}
	\caption{Illustrations of possible cases for the gradient $\nabla \funciso(x)$. $(a)$ The light gray stripes are the uncertainty sets for the directional derivatives. The dark gray squares are the uncertainty sets for $\nabla \funciso(x)$ and, the circles overapproximate the uncertainty sets.  In $(b)$ and $(c)$, the angle between the empirical gradient $\hat{\nabla \funciso(x)}$ and the dashed lines is the maximum angle between $\hat{\nabla \funciso(x)}$ and $\nabla \funciso(x)$.}
	\label{fig:gradsmallorbig}
\end{figure}
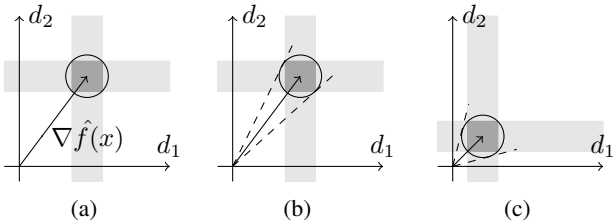

Assume that we sample the points that have a distance of $d$ from the center point in the isotropic coordinates. After $\dimension+1$ queries we can bound the gradient in a hypercube with  the edge length of $\smoothness d$. The hypercube can be contained in a hypersphere with radius $\smoothness \sqrt{\dimension} d /2$. For example, consider the case shown in Figure \ref{fig:hypercubeestimate}. Let $ \hat{\nabla \funciso(x)}$ be the empirical gradient estimate, i.e., the center of the hypercube. We can have two cases, either the gradient is in $ \cone(0,\arcsin(1/(2\dimension))$ or the magnitude of the gradient is smaller than $(2\dimension+1) \sqrt{\dimension}\smoothness d$. The former and latter cases are illustrated in Figures \ref{fig:largegradient} and \ref{fig:smallgradient}, respectively. In the latter case, if $d$ is sufficiently small, i.e., lower than $\suboptimality/((2\dimension+1)\sqrt{\dimension}\smoothness \radiusofset{\convexset})$, the center point is near optimal. Overall, the sample complexity of optimization using the value oracle with the ellipsoid method is $\tilde{\mathcal{O}}(\dimension^3)$. ~\cite{nemirovsky1983problem} provided a randomized optimization algorithm that succeeds with probability at least $1-\failureprob$ (where $\failureprob$ can be chosen to be arbitrarily small) and has a sample complexity of $\tilde{\mathcal{O}}(\dimension^{3})$ for Lipschitz continuous, convex functions. With an additional smoothness assumption, the method that we describe deterministically succeeds with the same complexity. We also remark that these bounds are inferior to the $\mathcal{O}(\dimension^{2})$ sample complexity result given in ~\cite{lee2018efficient}.

\subsubsection{Optimization using the noisy-value oracle} \label{section:optimizeNV}

For the noisy-value oracle, we can use the same gradient direction estimation method as in the value oracle. Different from the value oracle, we also need to consider the stochasticity of the oracle outputs since the empirical estimate $(\oracle{NV}(x) - \oracle{NV}(y))/\norm{x-y}$  of directional derivative is a $2 \standarddev^2 / \norm{x-y}^2 $-subgaussian random variable. 

We need $\tilde{\mathcal{O}}( \standarddev^2 /(\smoothness^2 \norm{x-y}^4))$ samples to obtain a confidence interval of $\mathcal{O}(\smoothness \norm{x-y})$  for the directional derivative estimate. By letting $\norm{x-y} = \mathcal{O}\left( \suboptimality/(2(2\dimension+1)\sqrt{\dimension}\smoothness \radiusofset{\convexset}) \right) $, we can ensure either that the ellipsoid method proceeds normally or that the current ellipsoid center is near optimal. Overall, the sample complexity of optimization using this method is $\tilde{\mathcal{O}}(\dimension^{13}/\suboptimality^{4})$. We remark that Belloni et al.~(\citeyear{belloni2015escaping}) provived an algorithm that has $\tilde{\mathcal{O}}(\dimension^{7.5}/\suboptimality^{2})$ sample complexity and $\suboptimality$-suboptimality in expectation.

%% file: regret.tex
\section{A sub-linear regret algorithm for the noisy-value oracle} \label{section:regret}
The regret of an optimization algorithm measures the performance of the algorithm during optimization. Define $x_{i}$ as the query point at time $i$, and $\hat{\func}(x_{i})$ as the output of the oracle. For a given number of queries $\timehorizon$, the regret of an algorithm $\algo$ is $\regret{\algo}(T) = \sum_{t=1}^{\timehorizon} \func\left(\mathcal{A}\left( h_{t}  \right)\right) - \sum_{t=1}^{\timehorizon} \func(x^{*}) $ where  $h_{t} = (x_{0}, \hat{\func}(x_{0})) \ldots (x_{t-1}, \hat{\func}(x_{t-1}))$ is the history of the algorithm. As in the previous section, we assume that $x^{*}$ is an interior point of $\convexset$ such that $\ellipsoid(\timehorizon^{-0.25} \identity / \lipschitz, x^{*}) \subseteq \convexset$.

The optimization algorithm mentioned in the previous section incurs sublinear regret when $\suboptimality = \mathcal{O}(\timehorizon^{-0.2})$. However, this approach yields a regret that has high order dependencies on the other parameters since the algorithm only relies on finding a near-optimal point with a regret of $\mathcal{O}(\timehorizon^{-0.2})$ if the gradient estimation fails. We give Algorithm \ref{algo:regretNV} that incurs $\tilde{\mathcal{O}}( \dimension^{3.75} \radiusofset{\convexset}  \sqrt{\smoothness \standarddev } \timehorizon^{0.75})$ regret with high probability when $\timehorizon = \Omega(\dimension^3 \lipschitz^{4/3} \standarddev^2  +\lipschitz^4 \standarddev^6)$ and $\dimension \radiusofset{\convexset}, \smoothness, \lipschitz, \standarddev \geq 1$. Different from  the  optimization algorithm, Algorithm \ref{algo:regretNV}  does not find a near-optimal point if the gradient estimation fails. Instead, Algorithm \ref{algo:regretNV} finds a point with a regret that incurs the half of the regret of the previous query point. While this approach increases the number of queries for optimization purposes, it yields a low regret.

 Algorithm \ref{algo:regretNV} consists of three phases. In Phase 1, we start by limiting the current convex set with the circumscribing ellipsoid and apply the isotropic transformation. We query the oracle in every dimension at the center of the ellipsoid and at the points that are close to the center. Then, we estimate the gradient within a confidence interval and limit the possible gradient directions to a cone in the isotropic coordinates. There are two possible cases:
 \begin{enumerate}
     \item If the semi-vertical angle of the possible gradient directions is small enough, i.e., less than $\arcsin(1/(2\dimension))$, we cut the current ellipsoid, and start the process from the beginning using the remaining set.
     \item If the semi-vertical angle of the possible gradient directions is not small enough, we halve the sampling distance and confidence interval, and start querying with the new sampling distance and confidence interval. 
 \end{enumerate}

If Case 2 happens, it implies that the gradient at the current ellipsoid center has a small magnitude, and the regret of the next set of queries is low. If Case 1 happens sufficiently many times, then one of the ellipsoid centers is a near optimal point with low regret as in the classical ellipsoid method. After Case 1 happens sufficiently many times, the algorithm proceeds to Phase 2. In this phase, we compare the the ellipsoid centers and find an ellipsoid center with a low regret, i.e., $\mathcal{O}(\timehorizon^{-0.25})$. In Phase 3, we repeatedly query the  ellipsoid center with a low regret.

\begin{algorithm} [ht]
	\caption{The low regret algorithm \textsc{Regret-NV($\timehorizon, \failureprob$)} for the noisy value oracle } \label{algo:regretNV}
	\begin{algorithmic}[1]
		\State  Set $\convexset^{(1)} = \convexset$. Find $\circumscribedellipsoid_{\convexset^{(1)}} = \ellipsoid(A^{(1)}, x^{(1)})$.  Set $X = \lbrace x^{(1)} \rbrace$.
		\State Set $K = \left \lceil 8\dimension(\dimension + 1) \log\left(  2 \radiusofset{\convexset} \lipschitz \timehorizon^{0.25} \right)  \right \rceil$, $\tau = \ceilx{ 32 \standarddev^2  \dimension^4 \log\left( \frac{2}{\failureprob'} \right)}$, $\failureprob' = \frac{\failureprob}{4 \dimension K \log_{16}\left( \frac{15 \timehorizon}{2 \dimension} \right)}$.
		\For{$ k =1, \ldots, K$} \Comment{Phase 1}
		\State Set $d = \frac{\min\left(\sqrt{\eigenvalue_{\max}(A^{(k)})}, 1\right)}{2 \dimension  }$.
		\State Set $\Delta = \frac{d  \left( 2+ \smoothness\eigenvalue_{\max}(A^{(k)}) \right)}{2\eigenvalue_{\max}(A^{(k)})}$.
		\For{$i = 0, 1, \ldots$} \Comment{Case 2}
		\State Set $d_{i} = d/2^{i}$, $\Delta_{i} = \Delta/2^{i}$, $\tau_{i} = 2^{4i} \tau$
		\State Query $\tau_{i}$ times $\oracle{NV}(x^{(k)})$ and $\oracle{NV}(\isotransinverse{A^{(k)}}{ x^{(k)}}(de_{j}))$  for all $i \in [\dimension]$. 
		\State For every query point $x$, set $\hat{\oraclewithoutname}^{NV}(x)$ as the mean of queries for point $x$.
		\State Estimate the gradient $p$ using the mean values $\hat{\oraclewithoutname}^{NV}(x)$.
		\If{$\left( \| p\| > \sqrt{\dimension} \Delta_{i} \right) \wedge \left(\arcsin\left(\frac{\sqrt{\dimension} \Delta_{i}}{\|p\|}\right) \leq \arcsin\left(\frac{1}{2\dimension}\right)\right)$} \Comment{Case 1}
		\State Set $\convexset^{(k+1)} = \convexset^{(k)} \cap  \ellipsoid(A^{(k)}, x^{(k)}) \cap \isotransinverse{A}{x^{(k)}} \left(\curlyset{ x | \innerproduct{p / \norm{p}}{ x } \leq 1/(2\dimension) }\right)$.
		\State Find $\circumscribedellipsoid_{\convexset^{(k+1)}} = \ellipsoid(A^{(k+1)}, x^{(k+1)})$. 
		\State Set $X = X \cup \lbrace x^{(k+1)} \rbrace$.
		\State \textbf{break}
		\EndIf
		\EndFor
		\EndFor
		\State For all $x\in X$, query $\oracle{NV}(x)$, $\ceilx{32 \standarddev^2 \sqrt{\timehorizon} \log\left( \frac{2(K+1)}{\failureprob}\right) }$ times. Set $x'$ to the point with the highest empirical mean.  \Comment{Phase 2}
		\State Repeatedly query $\oracle{NV}(x')$. \Comment{Phase 3}
	\end{algorithmic}
	
\end{algorithm}

\begin{theorem} \label{theorem:regret}
 Let $K = \left \lceil 8\dimension(\dimension + 1) \log\left(  2 \radiusofset{\convexset} \lipschitz \timehorizon^{0.25} \right)  \right \rceil$, $\failureprob' = \failureprob/ \left( 4 \dimension K \log_{16}\left( \frac{15 \timehorizon}{2 \dimension} \right) \right)$, and $\tau = \ceilx{ 8\standarddev^2 \smoothness^2 \dimension^4 \log\left( \frac{2}{\failureprob'} \right)}$. For an $\lipschitz$-Lipschitz, $\smoothness$-smooth, convex function $\func: \convexset \to \mathbb{R}$, a given failure probability  $\failureprob >0$, and a time horizon $\timehorizon$, Algorithm \ref{algo:regretNV} has a regret of at most $ K\left(\radiusofset{\convexset} \lipschitz \tau   + 5  \timehorizon^{0.75}  \dimension^{-0.25}\max\left( \dimension \radiusofset{\convexset}, 1 \right)(1+\smoothness)   \tau^{0.25} \right) + 
 (K+1) \ceilx{32 \standarddev^2 \sqrt{\timehorizon} \log\left( \frac{2(K+1)}{\failureprob}\right) } \radiusofset{\convexset} \lipschitz + \timehorizon^{0.75}$
 with probability at least $1-\failureprob$.
\end{theorem}

For a given $\lipschitz$-Lipschitz, $\smoothness$-smooth function $\func:\convexset \to \mathbb{R}$, we can define $\func':\convexset' \to \mathbb{R}$ such that $\convexset' = \curlyset{x' | x' = x\sqrt{\smoothness}, x \in \convexset}$ and $\func'(\sqrt{\smoothness}x) = \func(x)$ for all $x \in \convexset$.  $\func'$ is $\lipschitz/\sqrt{\smoothness}$-Lipschitz, $1$-smooth, and $ \radiusofset{\convexset'} = \sqrt{\smoothness} \radiusofset{\convexset}$. If Algorithm \ref{algo:regretNV} operates with the parameters of $\func'$, the regret is $\tilde{\mathcal{O}}(\dimension^{3.75} \radiusofset{\convexset} \sqrt{\smoothness \standarddev} \timehorizon^{0.75})$  when $\timehorizon = \Omega(\dimension^3 \lipschitz^{4/3} \standarddev^2  + \lipschitz^4 \standarddev^6 )$ and $\dimension \radiusofset{\convexset}\sqrt{ \smoothness}, \lipschitz , \standarddev \geq 1$.

%% file: conclusion.tex
\section{Discussion}
We consider the problem of minimizing a smooth, Lipschitz, convex function using sub-zeroth-order oracles. We leverage the smoothness property of the objective function and build variants of the ellipsoid method based on gradient estimation. We show that the sample complexities of optimization using sub-zeroth-order oracles are polynomial.

We remark that the main concern of this paper is the sample complexity of the optimization problems. The computational and space complexities of the provided algorithms are the same as those of the classical ellipsoid method. However, we remark that finding the exact minimum volume circumscribing ellipsoid for an arbitrary convex set is computationally intractable. In practice, we can avoid the computation of the minimum volume ellipsoids by finding an approximate enclosing ellipsoid using a separation oracle and analytical expressions involving the ellipsoid found at the previous step~\cite{goldfarb1982modifications}. We note that in the case of the comparator and noisy-value oracles, we use a property of the minimum volume circumscribing ellipsoid to give optimality and regret guarantees. We can show that a similar property holds for the approximate ellipsoids through additional feasibility cuts. 

For the directional-preference and comparator oracles, since the sampling distance can be arbitrarily reduced, through small modifications in the presented algorithms, the given sample complexities can be achieved with polynomial time complexities. For the noisy-value oracle we expect that polynomial time complexities can be achieved while maintaining a polynomial regret in the number of dimensions.

\section*{Acknowledgements}
This work was supported in part by AFOSR FA9550-19-1-0005, DARPA D19AP00004, NSF 1646522, and NSF 1652113.

%% file: supp_mat/supp_prelims.tex
\section{Preliminaries} \label{supp:section:prelims}
The unit vectors in $\Rdim$ are $e_{1}, \ldots, e_{\dimension}$. Let $S$ be a set of vectors in $\Rdim$. $Proj_{S}(x)$ denotes the orthogonal projection of $x$ onto the span of $S$ and $Proj_{S^{\bot}}(x)$ denotes the orthogonal projection of $x$ onto the complement space of the span of $S$. The angle between $x$ and $y$ is $\angle(x,y)$. $I$ denotes the identity matrix. The maximum and minimum eigenvalues of a square matrix $A$ is denoted by $\eigenvalue_{\max}(A)$ and $\eigenvalue_{\min}(A)$, respectively. The boundary of a set $D \in \Rdim$ is denoted by $Bd(D)$. The convex hull of a set $D$ of points is denoted by $Conv(D)$. With a slight abuse of notation, we use $0$ to denote the origin, i.e., $[0, \ldots, 0]^{\top} \in \Rdim$.

A convex function $\func: \convexset \to \mathbb{R}$ is said to be \textit{$\lipschitz$-Lipschitz} if $	\norm{\func(x) - \func(y)} \leq L \norm{x - y}$ for all $x,y \in \convexset$. A differentiable convex function $\func: \convexset \to \mathbb{R}$ is said to be \textit{$\smoothness$-strongly smooth} if $ |	\func(y) - \func(x) -  \innerproduct{\nabla \func(x)}{y-x}  | \leq  \smoothness \norm{y-x}^2 / 2 $ for all $x,y \in \convexset$.

	A \textit{right circular cone} in $\Rdim$ with \textit{semi-vertical angle} $\theta \in [0, \pi/2]$ and \textit{direction} $v \in \Rdim$ is $
	\cone(v, \theta) = \curlyset{w | W \in \Rdim, \angle(v,w) \leq \theta } $. 
	
	A \textit{ball} in $\Rdim$ is $\ball(r,x_{0}) =  \curlyset{x | \norm{x - x_{0}} \leq r} $
	where $x_{0} \in \Rdim$ and $r \geq 0$. The \textit{circumscribing ball} $\circumscribedball_{\convexset} = \ball(r^{*}, x^{*}_{0})$ of a compact convex set $\convexset$ satisfies $r^{*} = \min_{r^{*}, x_{0}} r$ where $\convexset \subseteq \ball(r, x_{0})$. The \textit{inscribed ball} $\inscribedball_{\convexset} = \ball(r^{*}, x^{*}_{0})$ of a compact convex set $\convexset$ satisfies $r^{*} = \max_{r^{*}, x_{0}} r$ where $\ball(r, x_{0}) \subseteq \convexset$. The \textit{radius} $\radiusofset{\convexset}$ of a compact convex set $\convexset$ is equal to the the radius of the circumscribing ball, i.e., $\radiusofset{\convexset} = \min_{y \in \convexset} \max_{x \in \convexset} \norm{x-y}.$  
	
	An \textit{ellipsoid} in $\Rdim$ is $		\ellipsoid(A,x_{0}) =  \curlyset{x | (x-x_{0})^{T} A^{-1} (x - x_{0}) \leq 1} $
	where $x_{0} \in \Rdim$ and $A \in \mathbb{R}^{n \times n}$ is a positive definite matrix. The \textit{isotropic transformation} $\isotrans{A}{x_{0}}$ of an ellipsoid $\ellipsoid(A, x_{0})$ is $\isotrans{A}{x_{0}}(x) = A^{-1/2}(x - x_{0})\sqrt{\eigenvalue_{\max}(A)}.$ The isotropic transformation repositions the ellipsoid at the origin and stretches the ellipsoid such that it becomes a hypersphere whose radius is equal to the largest radius of the ellipsoid. The inverse of $\isotrans{A}{x_{0}}$ is $\isotrans{A}{x_{0}}^{-1}(x) = A^{1/2}x/\sqrt{\eigenvalue_{\max}(A)} + x_{0}.$ With an abuse of notation we use $\isotrans{A}{x_{0}}(D)$ to denote the set $\curlyset{\isotrans{A}{x_{0}}(x) | x \in D }$. The \textit{circumscribing ellipsoid} $\circumscribedellipsoid_{\convexset} = \ellipsoid(A^{*}, x^{*}_{0})$ of a compact convex set $\convexset$ satisfies $\det(A^{*}) = \min_{A, x_{0}} \det(A)$ where $\convexset \subseteq \ellipsoid(A, x_{0}) $.	

	A  \textit{$\standarddev^2$-subgaussian} random variable $X$ with mean $\mu$ satisfies 
	$\Pr(|X - \mu| > t) \leq 2\exp\left( -t^2/(2\standarddev^2) \right)$ for all $t > 0$.

%% file: supp_mat/supp_proofs.tex
\section{Proofs for the technical results}
\setcounter{lemma}{0}
\setcounter{theorem}{0}
\setcounter{algorithm}{0}

We use Lemmas \ref{supp_lemma:ellipsoid} -- \ref{supp_lemma:feasibleinnerball} for the proofs of theorems given in the paper.

\begin{lemma} \label{supp_lemma:ellipsoid}
    Let $\func: \Rdim \to \mathbb{R}$ be a differentiable, convex function. For $\theta \in [0, \arcsin(1/\dimension)]$ and $p \in \Rdim$, if $\nabla \func(0) \in \cone(p, \theta)$, then $\func(x') \geq f(0)$ for all $x' \in \ellipsoid(I, 0) \cap 	 \curlyset{ x | \innerproduct{p / \norm{p}}{ x } > \sin\theta },$ and there exists an ellipsoid $\mathcal{E}^{*}$ such that $\mathcal{E}^{*} \supseteq \ellipsoid(I, 0) \cap 	 \curlyset{ x | \innerproduct{p / \norm{p}}{ x } \leq \sin\theta } $ and \[ \frac{Vol(\mathcal{E}^{*})}{Vol(\ellipsoid(I, 0)) } = \left( \frac{\dimension^2 (1- \sin^2(\theta))}{\dimension^2 - 1} \right)^{(\dimension-1)/2}  \frac{\dimension(1+\sin(\theta))}{\dimension+1} \]  If $\theta = \arcsin(1/(2\dimension))$, then \[Vol(\mathcal{E}^{*}) \leq Vol(\ellipsoid(I, 0)) e^{-\frac{1}{8(\dimension + 1)}} <Vol(\ellipsoid(I, 0)).\]
\end{lemma} 

\begin{proof} [Proof of Lemma \ref{supp_lemma:ellipsoid}]
	We first show that if $\nabla \func(x') \in \cone(p, \theta)$, then $\func(x') \geq f(x)$ for all $x \in D = \ellipsoid(I, x') \cap 	 \curlyset{ x | \innerproduct{p / \norm{p}}{ x } \leq \sin\theta }.$ By the convexity of $\func$, we have $\func(0) \geq \func(x) - \innerproduct{\nabla \func(0)}{x}$ for all $x \in \Rdim$. Since $\nabla \func(0) \in \cone(p, \theta)$ and $\innerproduct{x}{y} \geq 0 $ for all  $x \in \cone(p, \pi/2- \theta)$ and $y \in \cone(p, \theta)$, we have $\func(0) \leq \func(x) - \innerproduct{\nabla \func(0)}{x} \leq \func(x)$ for all $x \in \cone(p, \pi/2- \theta)$. Since $\ellipsoid(I, 0) \cap 	 \curlyset{ x | \innerproduct{p / \norm{p}}{ x } > \sin\theta } \subset \cone(p, \pi/2- \theta)$, we have $\func(0) \leq \func(x) - \innerproduct{\nabla \func(0)}{x} \leq \func(x)$ for all $x \in\ellipsoid(I, 0) \cap 	 \curlyset{ x | \innerproduct{p / \norm{p}}{ x } > \sin\theta }$.
	
	By Theorem 2.1 of \cite{goldfarb1982modifications}, there exists an ellipsoid $\mathcal{E}^{*}$ such that $\ellipsoid^{*} \supseteq \ellipsoid(I, 0) \cap 	 \curlyset{ x | \innerproduct{p / \norm{p}}{ x } \leq \sin\theta }$ and \[ Vol(\mathcal{E}^{*}) = Vol(\ellipsoid(I, 0)) \left( \frac{\dimension^2 (1- \sin^2(\theta))}{\dimension^2 - 1} \right)^{(\dimension-1)/2}  \frac{\dimension(1+\sin(\theta))}{\dimension+1} < Vol(\ellipsoid(I, 0)). \] 
	Setting $\theta = \arcsin(1/(2\dimension))$, we get \[ \frac{Vol(\mathcal{E}^{*})}{Vol(\ellipsoid(I, 0))}  =  \left( \frac{4\dimension^2 - 1}{4\dimension^2 - 4} \right)^{(\dimension-1)/2}  \frac{2\dimension+1}{2\dimension+2} = \left( 1+ \frac{3}{4\dimension^2 - 4} \right)^{(\dimension-1)/2} \left( 1-  \frac{1}{2\dimension+2} \right) . \] By the inequality $1+x \leq e^{x}$, we have \[ \frac{Vol(\mathcal{E}^{*})}{Vol(\ellipsoid(I, 0))}  \leq  e^{\frac{3(\dimension-1)}{2(4\dimension^2 - 4)}} e^{-\frac{1}{2\dimension + 2}} = e^{-\frac{1}{8(\dimension+1)}}  . \] 
\end{proof}

\begin{lemma} \label{supp_lemma:smallergradientcones}
	Let $\gamma \in (0, \pi/2],$ $d_{1} = e_{1},$ $d_{i} = \cos(\gamma)e_{1} + \sin(\gamma)e_{i},$ for all $ i\in \lbrace 2, \ldots, \dimension \rbrace$, $p = \sum_{i=1}^{\dimension} d_{i},$ and $\gamma' = \arccos(\langle p,d_{2} \rangle/ \|p\|) )$. 	Then, \(\mathcal{F} (p, \gamma' ) \supseteq \mathcal{F}(d_{1}, \gamma ) \cap \lbrace x | x_{i} \geq 0 \rbrace\) and \(\sin(\gamma')/\sin(\gamma) \leq \sqrt{\dimension-1} / \sqrt{\dimension}. \)
\end{lemma}

\begin{proof}[Proof of Lemma \ref{supp_lemma:smallergradientcones} ] 
	We first show that $\mathcal{F} (p, \gamma' ) \supseteq \mathcal{F} (d_{1}, \gamma ) \cap \lbrace x | x_{i} \geq 0 \rbrace. $ It suffices to show that the semi-vertical angle $\gamma'$ of the new cone is larger than the angle between the direction $q$ of the new cone and any enclosed point. Formally, we need to show that $\gamma' \geq \max \arccos\left( \frac{\langle p, q \rangle}{ \|p\|\|q\|}  \right)$ where $q \in \mathcal{F} (d_{1}, \gamma ) \cap \lbrace x | x_{i} \geq 0 \rbrace.$ 
	
	Without loss of generality assume that $q = a d_{1} + \sum_{i=2}^{n} \sqrt{1-a^2} b_{i} d_{i} $ where $ 0 \leq a \leq 1$, $\sum_{2}^{n} b_{i}^2 = 1$, and $ 0 \leq b_{i} \leq 1$ for all $ i \in \curlyset{2, \ldots, n}.$ Note that this assumption only limits the scaling of $q$ such that $\norm{q} = 1$ and does not affect the maximum angle.
	
	We have \[\frac{\innerproduct{p}{q}}{\norm{p} \norm{q}} = \frac{a ((n-1) \cos(\gamma) + 1)}{n} + \left(\sqrt{1-a^2} \cos(\gamma) \frac{(n-1)\cos(\gamma) + 1}{n} \right)  \sum_{i=2}^{n}  b_{i}   + \frac{\sin^2(\gamma)}{n} \sum_{i=2}^{n} b_{i}. \]
	
	For a fixed value of $a$, $\frac{\innerproduct{p}{q}}{\norm{p} \norm{q}}$ is minimized, i.e., $\arccos\left( \frac{\innerproduct{p}{q}}{\norm{p} \norm{q}} \right)$ is maximized, when $b_{i} = 1$ for some $i \in \curlyset{2, \ldots, n}$ and $b_{j} = 0$ for others. In order to find the maximum value of $\arccos\left( \frac{\innerproduct{p}{q}}{\norm{p} \norm{q}} \right)$, without loss of generality we assume that $b_{2} = 1$, $b_{j} = 0$ for all $j \in \curlyset{2, \ldots, n}$. Therefore, there exists $q = a d_{1} + \sqrt{1-a^2} d_{2} $ such that $\arccos\left( \frac{\innerproduct{p}{q}}{\norm{p} \norm{q}} \right)$ is maximized.
	
	Define $q'$  such that $q' = b d_{1} + (1-b) d_{2}$ where $0 \leq b \leq 1$ and  $\arccos\left( \frac{\innerproduct{p}{q}}{\norm{p} \norm{q}} \right) = \arccos\left( \frac{\innerproduct{p}{q'}}{\norm{p} \norm{q'}} \right)$. Note that $q'$ is a scaled version of $q$, i.e., $q = q'/\norm{q'}.$
	
	We note that \[ \max_{q'} \arccos\left( \frac{\innerproduct{p}{q'}}{\norm{p} \norm{q'}} \right) \leq \max_{q'} \arccos\left(\frac{\innerproduct{p}{q'}}{\norm{p} } \right) \] since $\arccos(\alpha)$ is a non-increasing function of $\alpha$. We also note that $\arccos\left(\frac{\innerproduct{p}{q'}}{\norm{p} } \right)$ is maximized when $\innerproduct{p}{q'}$ is minimized and $\innerproduct{p}{q'}$ is a linear function of $b$ on the compact, convex set $0 \leq b \leq 1$. Therefore, there exists a corner point $b \in \curlyset{0, 1}$ such that  $\arccos\left(\frac{\innerproduct{p}{q'}}{\norm{p} } \right)$ is maximized.
	
	For $b =1$, we have $q' = q = d_{1}$ and  \[\langle p, d_{1} \rangle= \frac{1+(\dimension-1)\cos(\gamma)}{\dimension}. \] For $b =0$, we have $q' = q = d_{2}$ and \[\langle p, d_{2} \rangle = \frac{\cos(\gamma)+(\dimension-1)\cos^2(\gamma)+\sin^2(\gamma)}{\dimension} = \frac{1+ \cos(\gamma)+ (\dimension-2)\cos^2(\gamma)}{\dimension}  \] for all $i \in \lbrace 2, \ldots, \dimension\rbrace.$ Note that $\langle p, d_{2}\rangle \leq \langle p, d_{1}\rangle$ since $\cos(\gamma) \leq 1$.
	
	We consequently have $ \arccos\left( \frac{\langle p, d_{1} \rangle}{ \|p\|}  \right) \leq \arccos\left( \frac{\langle p, d_{2} \rangle}{ \|p\|}  \right)$ and $\arccos\left( \frac{\langle p, q \rangle}{ \|p\|}  \right)$ is maximized when $q = d_{2}$.	Therefore, \[\gamma' = \arccos\left( \frac{\langle p, d_{2} \rangle}{ \|p\|}  \right) =  \max_{q'} \arccos\left( \frac{\langle p, q \rangle}{ \|p\|}  \right) \geq \max_{q'} \arccos\left( \frac{\langle p, q \rangle}{ \|p\|\|q\|}  \right) =\max_{q} \arccos\left( \frac{\langle p, q \rangle}{ \|p\|\|q\|}  \right) \]  which implies that \[ \mathcal{F} (p, \gamma' ) \supseteq \mathcal{F} (d_{1}, \gamma ) \cap \lbrace x | x_{i} \geq 0 \rbrace.\] 
	
	We now prove that $\frac{\sin(\gamma')}{\sin(\gamma)} \leq \sqrt{\frac{\dimension-1}{\dimension}}.$ We have \begin{subequations}
		\begin{align*}
		\frac{\sin(\gamma')}{\sin(\gamma)} &= \frac{\sin \left( \arccos\left( \frac{\langle p, d_{2} \rangle}{\|p\| \|d_{2}\|} \right)  \right)}{\sin(\gamma)}
		\\
		&=  \frac{\sqrt{1 - \frac{\left( 1+ \cos(\gamma)+ (\dimension-2)\cos^2(\gamma) \right)^2}{\dimension^2 \left( \left(\frac{1+(\dimension-1)\cos(\gamma)}{\dimension}\right)^2 + (\dimension-1)\left( \frac{\sin(\gamma)}{\dimension} \right)^2 \right)}}}{\sin(\gamma)}
		\\
		&= \sqrt{\frac{(\dimension-2)^2 \cos^2(\gamma) + 2(\dimension-2)\cos(\gamma) + \dimension -1}{(\dimension-1)(\dimension-2)\cos^2(\gamma) + 2(\dimension-1)\cos(\gamma)+\dimension}}.
		\end{align*}
	\end{subequations}
	
	For $\gamma \in (0, \pi/2)$, we have 		
	\begin{subequations}
		\begin{align*}
		\frac{\partial }{\partial \gamma } \frac{\sin(\gamma')}{\sin(\gamma)} &= \frac{\sin(\gamma)((\dimension-2)\cos(\gamma) + 1)}{((\dimension-1)(\dimension-2)\cos^2(\gamma) + 2(\dimension-1)\cos(x) + \dimension)^2 \sqrt{\frac{(\dimension-2)^2 \cos^2(\gamma) + 2(\dimension-2)\cos(\gamma) + \dimension -1}{(\dimension-1)(\dimension-2)\cos^2(\gamma) + 2(\dimension-1)\cos(\gamma)+\dimension}}}
		\\
		&\geq 0,
		\end{align*}
	\end{subequations}
	i.e., $\frac{\sin(\gamma')}{\sin(\gamma)}$ is a non-decreasing function of $\gamma$. 
	
	Since  $\frac{\sin(\gamma')}{\sin(\gamma)} = \sqrt{\frac{\dimension-1}{\dimension}}$ when $\gamma = \pi/2$ and   $\frac{\sin(\gamma')}{\sin(\gamma)}$ is a non-decreasing function of $\gamma$, we conclude that $\frac{\sin(\gamma')}{\sin(\gamma)} \leq  \sqrt{\frac{\dimension-1}{\dimension}}.$
\end{proof}

\begin{lemma}\label{supp_lemma:ratioofellipsoidandcircle}
	Let $\convexset \in \Rdim$ be a compact convex set. The circumscribing ellipsoid $\circumscribedellipsoid_{\convexset} =  \ellipsoid(A^{*}_{\ellipsoid}, x^{*}_{0, \ellipsoid})$ and the radius $\radiusofset{\convexset}$ of $\convexset$ satisfies $\sqrt{\eigenvalue_{\max}(A^{*}_{\ellipsoid})} \leq \dimension \radiusofset{\convexset}.$ 
\end{lemma}

\begin{proof}[Proof of Lemma \ref{supp_lemma:ratioofellipsoidandcircle}]
	Let $\convexset_{0}$ be the convex set that is the isotropic transformation of $\convexset$, i.e., $\convexset_{0} = \lbrace x | \isotrans{A^{*}}{x^{*}_{0}}^{-1}(x) \in \convexset \rbrace$. Since $\ellipsoid(A^{*}_{\ellipsoid}, x^{*}_{0, \ellipsoid})$ is the circumscribing ellipsoid of $\convexset$, the circumscribing ellipsoid of $\convexset_{0}$ is  $\ball(\sqrt{\eigenvalue_{\max}(A^{*}_{\ellipsoid})},0))$ and equal to the circumscribing ball $\circumscribedball_{\convexset_{0}}$ of $\convexset_{0}$. Let $\inscribedball_{\convexset_{0}} = \ball(r,x_{0})$ be the inscribed ball of $\convexset_{0}$. Since $\convexset_{0}$ is convex, we have that $\sqrt{\eigenvalue_{\max}(A^{*}_{\ellipsoid})} \leq \dimension r$~\cite{henk2012lowner}.
	
	We note that the transformation $\isotrans{A^{*}}{x^{*}_{0}}$ preserves the distances between two point if the line passing through the points is parallel to the eigenvector that is associated with the largest eigenvalue of $A^{*}$. Since there exist points $x,y \in \ball(r,x_{0}) \subseteq C_{0}$ such that $\norm{x-y} = r$ and $x-y$ is parallel to the eigenvector that is associated with the largest eigenvalue of $A^{*}$, there exist two points in $\convexset$ such that the distance between the points is $r$. Therefore, the radius $\radiusofset{\convexset}$ of $\convexset$ satisfies  $\radiusofset{\convexset} \geq r$.
	
	By combining $\sqrt{\eigenvalue_{\max}(A^{*}_{\ellipsoid})} \leq \dimension r$ and $\radiusofset{\convexset} \geq r$, we get $\sqrt{\eigenvalue_{\max}(A^{*}_{\ellipsoid})} \leq \dimension \radiusofset{\convexset}.$

\end{proof}

\begin{lemma} \label{supp_lemma:feasibleinnerball}
	Let $\convexset \in \Rdim$ be a compact convex set and $\circumscribedellipsoid_{\convexset}  = \ellipsoid(A^{*}, x^{*}_{0})$ be the circumscribing ellipsoid of $\convexset$. If $x \in \ball(\eigenvalue_{\max}(A)/(2n), 0)$, then $\isotrans{A^{*}}{x^{*}_{0}}^{-1}(x) \in \convexset.$
\end{lemma}

\begin{proof}[Proof of Lemma \ref{supp_lemma:feasibleinnerball}]
	We prove the statement by contradiction: If there exists an $x \in \ball(\eigenvalue_{\max}(A)/(2n), 0)$, such that $T^{-1}_{A^{*}, x^{*}_{0}}(x) \not\in \convexset$, then $\ellipsoid(A^{*}, x^{*}_{0})$ is not the circumscribing ellipsoid of $\convexset$.
	
	Let $\convexset_{0}$ be the convex set that is the isotropic transformation of $\convexset$, i.e., $\convexset_{0} = \lbrace x | T^{-1}_{A^{*}, x^{*}_{0}}(x) \in \convexset \rbrace$. Since the ratios of volumes is constant for affine transformations, the circumscribing ellipsoid of $\convexset_{0}$ is $\ball(\sqrt{\eigenvalue_{\max}(A)}, 0)$. 
	
	Let $\ball(r, 0)$ be the ball with the maximum radius centered at the origin such that $\ball(r, 0) \in \convexset_{0}$. Then, there must exists a point $x'$ such that $\|x'\| = r$ and $x' \in Bd(\convexset_{0})$.

	By the supporting hyperplane theorem~\cite{boyd2004convex} there exists a supporting hyperplane at $x$ such that the entire convex set $\convexset_{0}$ is on one side of the hyperplane. Let $\halfspace = \lbrace x | \langle h, (x - x') \rangle \leq 0 \rbrace$ be the halfspace that contains $\convexset_{0}$ and passes through $x'$. Assume that the hyperplane $\langle h, (x - x') \rangle = 0$ is not tangent to $\ball(r, 0)$, i.e., $h$ is not a multiple of $x'$, then we have $\halfspace \cap \ball(r, 0) \neq \emptyset$ and $(\Rdim \setminus \halfspace) \cap \ball(r, 0) \neq \emptyset$. Since $\ball(r, 0) \subseteq \convexset_{0}$, we also have $\halfspace \cap \convexset_{0} \neq \emptyset$ and $(\Rdim \setminus \halfspace) \cap \convexset_{0} \neq \emptyset$. Therefore, the supporting hyperplane must be tangent to $\ball(r, 0)$ at $x'$, i.e., $h$ must be a multiple of $x'$. Also since $\ball(r, 0) \subset \halfspace = \lbrace x | \langle h, (x - x') \rangle \leq 0 \rbrace$, $h$ must be a positive multiple of $x'$. Without loss of generality assume that $h = x'$. 
	
	We have $\convexset_{0} \subseteq \halfspace = \curlyset{x | \langle x', (x - x') \rangle \leq 0} $ where $\norm{x'} = r$. Assume that $r < \sqrt{\eigenvalue_{\max}(A)}/(2n)$. In this case, by Lemma \ref{supp_lemma:ellipsoid}, there exists a an ellipsoid whose volume is smaller than $\volumeofunitball \eigenvalue_{\max}(A)^{n/2}$. This leads to a contadiction as we know that the circumscribing ellipsoid of $\convexset_{0}$ is $\ball(\sqrt{\eigenvalue_{\max}(A)}, 0)$. Therefore, $r \geq \sqrt{\eigenvalue_{\max}(A)}/(2n)$.
	
	Since $\ball(r, 0) \subseteq \convexset_{0}$ and $r \geq \sqrt{\eigenvalue_{\max}(A)}/(2n) $, we have that if $x \in \ball(\sqrt{\eigenvalue_{\max}(A)}/(2n), 0)$, then $\isotrans{A^{*}}{x^{*}_{0}}^{-1}(x) \in \convexset.$
\end{proof}

\newpage

\subsection{Proof of Theorem \ref{supp_theorem:DP}}
To prove Theorem \ref{supp_theorem:DP}, we use Algorithm \ref{supp_algo:optimizeDP} which is similar to the optimization algorithm under the comparator oracle. Algorithm \ref{supp_algo:optimizeDP} estimates the gradient direction by at the current ellipsoid center by querying the directional derivatives function in different orthogonal directions. After the estimation of the gradient, Algorithm \ref{supp_algo:optimizeDP} proceeds to the ellipsoid cut. Before the algorithm terminates Algorithm \ref{supp_algo:optimizeDP} compares the ellipsoid centers and outputs a point that is near optimal. For the comparison step, we employ the function \textsc{Compare-DP}  which uses bisection search to find a near optimal point from a given set of points.

\begin{algorithm}[ht] 
	\caption*{\textbf{Function} \textsc{Compare-DP($X, \suboptimality$)}  } \label{supp_algo:DPcomparison}
	\begin{algorithmic}[1]
		\State Set $X^{*} = X$ and $m = |X|$.
		\While{$|X^{*}| > 1$}
		\State Arbitraritly pick $x^{1}, x^{2} \in X$ such that $x^{1} \neq x^{2}$.
		\State Set $X^{*} = X^{*} \setminus \lbrace x^{1}, x^{2} \rbrace$.
		\State Set $x^{l} = x^{1}$ and $x^{r} = x^{2}$.
		\While{$\| x^{r} - x^{l}\| \leq  2\suboptimality/(\lipschitz m)$}
		\State Query $\oracle{DP}((x^{r} + x^{l})/2, (x^{r} - x^{l})/2)$.
		\If{$\oracle{DP}((x^{r} + x^{l})/2, (x^{r} - x^{l})/2) = 0$}
		\State $x^{l} = (x^{r} + x^{l})/2.$
		\Else
		\State $x^{r} = (x^{r} + x^{l})/2.$
		\EndIf
		\EndWhile
		\State $X^{*} = X^{*} \cup \lbrace (x^{r} + x^{l})/2 \rbrace$
		\EndWhile
		\State \textbf{return} $x^{*} \in X^{*}$.	
	\end{algorithmic}
	
\end{algorithm}

\begin{lemma} \label{supp_lemma:compareDP}
	For an $\lipschitz$-Lipschitz function $f:\convexset \to \mathbb{R}$ and a set $X$ of points with size $m$, The function \textsc{Compare-DP} makes at most $(m-1) \log_{2} \frac{\radiusofset{\convexset} \lipschitz m}{2\suboptimality}$ queries to $\oracle{DP}$ and the output $X^{*}$ of the above algorithm satisfies $f(x^{*}) \leq \min_{x \in X} f(x) + \suboptimality$, $x^{*} \in \convexset,$ and $x^{*} \in X^{*}$. 
\end{lemma}

\begin{proof}[Proof of Lemma \ref{supp_lemma:compareDP}]
	The proof follows from bisection search. We observe that in every iteration of the inner while loop the algorithm halves the search space $Conv(\curlyset{x^{l}, x^{r}})$ according to the result of the directional derivative at the mid point $(x^{r} + x^{l})/2$.  We also note that since only the ascent directions are discarded, at the end of the inner while loop, there exists a point $x^{*} \in Conv(\curlyset{x^{l}, x^{r}})$ such that $\func(x^{*}) = \min_{Conv(\curlyset{x^{1}, x^{2}})}$. Since $\func$ is $\lipschitz$-Lipschitz, and $\norm{x^{r} - x^{l}} \leq 2\suboptimality/(\lipschitz m)$, we have $\func((x^{r} + x^{l})/2) \leq \func(x^{*}) + \suboptimality/m$. 
	
	At the beginning of the inner while loop, the algorithm removes two points $x^{1}, x^{2}$ from $X$ and at the end of the inner while loop the algorithm adds a point $x'$ such that $\func(x') \leq  \min(\func(x^{1}), \func(x^{2}) + \suboptimality/m)$. Therefore, in each iteration of the outer while loop the size of $X^{*}$ decreases by $1$ and the minimum function value among the points in $X^{*}$ increases by at most $\suboptimality/m$. Since the outer loop makes at most $m-1$ iterations, the output point $x^{*}$ satisfies $\func(x^{*}) \leq \min_{x \in X} \func(x) + \suboptimality$.
	
	Since $\norm{x^{1} - x^{2}} \leq \radiusofset{\convexset}$ for all $x^{1}, x^{2} \in X^{*}$, and $\norm{x^{r} - x^{l}}$ is halved in each iteration, the inner while loop makes at most $\log_{2}\frac{\radiusofset{\convexset} \lipschitz m} {2 \suboptimality}$ iterations. Since the outer loop makes at most $m-1$ iterations the number of queries is bounded by $(m-1) \log_{2}\frac{\radiusofset{\convexset} \lipschitz m}{2 \suboptimality}.$

\end{proof}

\begin{algorithm}[ht] 
	\caption{The optimization algorithm \textsc{Optimize-DP($X, \oracle{DP}$)} for the directional preference oracle } \label{supp_algo:optimizeDP}
	\begin{algorithmic}[1]
		\State Find $\circumscribedellipsoid_{\convexset} = \ellipsoid(A^{(k)}, x^{(1)})$ of $\convexset$.
		\State Set $X = \lbrace x^{(1)} \rbrace$, $\convexset^{(1)} = \convexset$, $K = \left \lceil 8\dimension(\dimension+1) \log\left( \frac{2 \radiusofset{\convexset} \lipschitz}{\suboptimality}  \right) + 1  \right \rceil$.
		\For{$ k =1 \ldots K$}
		\State Set $p =$ \textsc{PD-DP}$\left( \oracle{DP}, x^{(k)}, \arcsin\left( 1/(2n)\right), A^{(k)} \right)$.
		\State Set $\convexset^{(k+1)} = \convexset^{(k)} \cap  \ellipsoid(A^{(k)}, x^{(k)}) \cap \isotransinverse{A^{(k)}}{x^{(k)}} \left(\curlyset{ x | \innerproduct{p / \norm{p}}{ x } \leq 1/(2\dimension) }\right)$.
		\State Find $\circumscribedellipsoid_{\convexset^{(k+1)}} = \ellipsoid(A^{(k+1)}, x^{(k+1)})$ of $\convexset^{(k+1)}$.
		\State Set $X = X \cup \lbrace x^{(k+1)} \rbrace$.
		\EndFor
		\State \textbf{return} $\textsc{Compare-DP}(X,\oracle{DP}, \suboptimality/2)$.
		
	\end{algorithmic}
	
\end{algorithm}

\begin{algorithm}[ht] 
	\caption*{\textbf{Function} \textsc{PD-DP}($ x, \theta, \isotrans{A}{x} $)} 
	\begin{algorithmic}[1]
	\State $p = e_{1}$, $r = 1$, $\gamma = \pi/2$.
	\While{$\gamma > \theta$}
	\State Find $d_{i}$ such that $d_{1} = p$, $d_{i} \perp d_{j}$ for all $i\neq j \in [n]$, and $\| d_{i} \| = 1$ for all $i \in [n]$. 
	\State Query $\oracle{DP}(x, A^{-1/2}d_{1}), \ldots, \oracle{DP}(x,A^{-1/2}d_{n})$. \label{supp_algo:findperpdirections}
	\State Set $w_{1} = d_{1}$ and for all $i \in \lbrace 2, \ldots, n \rbrace$, set $w_{i} =  d_{1} \cos(\gamma) +  d_{i} \oracle{DP}(x, A^{-1/2}d_{i}) \sin(\gamma)$. \label{supp_algo:projection}
	\State Set $p = \left( \sum_{i=1}^{n} w_{i}/n \right)/ \left\| \sum_{i=1}^{n} w_{i}/n  \right \|$, \label{supp_algo:findthesmallercone}
	\State Set $\gamma = \arccos(\langle p, w_{2} \rangle)$.
	\State Set $r = \arcsin (\gamma)$. \label{supp_algo:assignment}
	\EndWhile
	\State \textbf{return} $p$.
	\end{algorithmic}
\end{algorithm}

\begin{theorem} \label{supp_theorem:DP}
    Let $K = \ceilx{8 \dimension(\dimension +1) \log\left( \frac{2 \radiusofset{\convexset \lipschitz}}{\suboptimality} \right)}.$ For an $\lipschitz$-Lipschitz, $\smoothness$-smooth, convex function $\func: \convexset \to \mathbb{R}$, Algorithm \ref{supp_algo:optimizeDP} makes at most
    \[\dimension K \ceilx{2\dimension \log(2\dimension)}  + K \log_{2}\left( \frac{\radiusofset{\convexset} \lipschitz (K + 1) }{\suboptimality} \right) \] queries to $\oracle{DP}$ and the output $x'$ of Algorithm \ref{supp_algo:optimizeDP} satisfies $\func(x') \leq \min_{x \in \convexset} \func(x) + \suboptimality$.
\end{theorem}

\begin{proof}[Proof of Theorem \ref{supp_theorem:DP}]
	We prove the theorem by showing that the output $x'$ of Algorithm \ref{supp_algo:optimizeDP} satisfies $\func(x') \leq \min_{x \in \convexset} \func(x) + \suboptimality$ and Algorithm \ref{supp_algo:optimizeDP} makes at most $\dimension K \ceilx{2\dimension \log(2\dimension)}  + K \log\left( \frac{\radiusofset{\convexset} \lipschitz (K+1) }{\suboptimality} \right)$ queries to $\oracle{DP}$ where $K =  \ceilx{8 \dimension(\dimension +1) \log\left( \frac{2 \radiusofset{\convexset \lipschitz}}{\suboptimality} \right)}$.
	
	We first show that the output $x'$ of Algorithm \ref{supp_algo:optimizeDP} satisfies $\func(x') \leq \min_{x \in \convexset} \func(x) + \suboptimality$. We note that due to Lemma \ref{supp_lemma:smallergradientcones}, at iteration $k$, the cone $T^{-1}_{A,x^{k}} \left(\cone(p, \arcsin(1/(2\dimension)))\right)$  of possible gradient directions after the gradient pruning algorithm terminates, includes the gradient. Consequently, the dual cone of $T^{-1}_{A,x^{k}} \left(\cone(p, \arcsin(1/(2\dimension)))\right)$ includes only the non-descent directions, i.e., $\func(x) \geq \func(x^{k}) $ for all $x \in T^{-1}_{A,x^{k}} \left(\cone(p, \pi/2 - \arcsin(1/(2\dimension)))\right)$. Therefore, after iteration $k$ there exists a $x^{*} \in \convexset^{(k+1)}$ such that $\func(x^{*}) = \min_{x \in \convexset} \func(x)$. 
	
	Since $\func$ is $L$-Lipschitz, the volume of the set $\curlyset{x | x \in \convexset, \func(x) \leq \func(x^{*}) + \suboptimality/2}$ is at least $\volumeofunitball \left( \frac{\suboptimality}{2 \lipschitz}\right)^{\dimension}$. Due to Lemma \ref{supp_lemma:ellipsoid}, we have $Vol\left(\circumscribedellipsoid_{\convexset^{(K+1)}} \right)< \volumeofunitball \left( \frac{\suboptimality}{2 \lipschitz}\right)^{\dimension}$. Therefore, there exists a point $x$ such that $x \not \in \convexset^{(K+1)}$ and $\func(x) \leq \func(x^{*}) + \suboptimality/2$

	Since every discarded point $x$ satisfies $\func(x) \geq \func(x^{k}) $ for some $1 \leq k \leq K$, we have $\func(x^{k}) \leq \func(x^{*}) + \suboptimality/2$ for some $1 \leq k \leq K$. 	Due to Lemma \ref{supp_lemma:compareDP}, the output point $x' = \textsc{Compare-DP}(X,\oracle{DP}, \suboptimality/2)$ satisfies $\func(x') \leq \min_{x \in X} \func(x) + \suboptimality/2 \leq \func(x^{*}) + \suboptimality.$ 
	
	We now prove the bound on the number of queries. The gradient pruning algorithm starts with $\gamma = \pi/2$. As shown in Lemma \ref{supp_lemma:smallergradientcones}, we have $\sin(\gamma) \leq \sqrt{ \frac{\dimension-1}{\dimension}}^{k} \leq e^{\frac{-k}{2\dimension}}$ after $k$ iterations. Since $\theta = \arcsin(1/2\dimension)$, the gradient pruning algorithm \textsc{PD-DP} stops after at most $\left \lceil 2\dimension\log(2\dimension) \right \rceil$ iterations where we make $n$ queries in each iteration. The for loop in Algorithm \ref{supp_algo:optimizeDP} has $K$ iterations. Therefore, the total number of queries due to the gradient pruning algorithm is $\dimension \left \lceil 2\dimension\log(2\dimension) \right \rceil K.$
	
	When $\textsc{Compare-DP}$ is called in Algorithm \ref{supp_algo:optimizeDP}, the set has $X$ has $K+1$ elements. Due to Lemma \ref{supp_lemma:compareDP}, the process $\textsc{Compare-DP}(X,\oracle{DP}, \suboptimality/2)$ makes $K \log_{2} \frac{\radiusofset{\convexset} \lipschitz (K+1)}{\suboptimality}$ queries. 
	
	The total number of queries is bounded by $\dimension K \ceilx{2\dimension \log(2\dimension)}  + K \log\left( \frac{\radiusofset{\convexset} \lipschitz (K+1) }{\suboptimality} \right)$ .
	 
\end{proof}
\subsection{Proof of Theorem \ref{supp_theorem:C}}
The proof of Theorem \ref{supp_theorem:C} is similar to the proof of Theorem \ref{supp_theorem:DP}. If the direction pruning algorithm does not encounter an unknown direction, the algorithm approximately estimates the direction of the gradient. If there is an unknown direction, then we consider two cases: the magnitude of the projection of the gradient in the known directions is large compared to the magnitude of the projection of the gradient in the known directions and otherwise. We show that, in the first case, the estimated gradient direction is still close to the direction of the gradient. In the second case, we show that the ellipsoid center is near-optimal since the magnitude of the gradient is small.  

\begin{algorithm}[ht] 
	\caption{The optimization algorithm \textsc{Optimize-C($\suboptimality$)} for the comparator oracle } \label{supp_algo:optimizeC}
	\begin{algorithmic}[1]
		\State Set  $\convexset^{(1)} = \convexset$. Find $\circumscribedellipsoid_{\convexset^{(1)}} = \ellipsoid(A^{(1)}, x^{(1)})$ of $\convexset^{(1)}$.
		\State Set  $X = \lbrace x^{(1)} \rbrace$, $K = \left \lceil 8\dimension(\dimension+1) \log\left( \frac{ \radiusofset{\convexset} \lipschitz}{\suboptimality}  \right)  \right \rceil$, $\kappa = \max\left( \frac{4}{4 \dimension - \sqrt{2}\dimension\sqrt{\frac{4\dimension^2 - 1}{4\dimension^2}}} , 1 \right).$
		\For{$ k =1 \ldots K$}
		\State Set $t^{(k)} = \frac{\min(\suboptimality, \sqrt{\eigenvalue_{\max}(A^{k})} )}{\kappa \dimension^{5/2} \max(\smoothness, 1) \max(\radiusofset{\convexset}, 1)}.$ 
		\State Set $p$ = \textsc{PD-C}$\left(  x^{(k)}, \arcsin\left(\frac{1}{2\sqrt{2}\dimension}\right), A^{(k)}, t^{(k)} \right).$
		\State Set $\convexset^{(k+1)} = \convexset^{(k)} \cap  \ellipsoid(A^{(k)}, x^{(k)}) \cap \isotransinverse{A^{(k)}}{x^{(k)}} \left(\curlyset{ x | \innerproduct{p / \norm{p}}{ x } \leq 1/(2\dimension) }\right)$.
		\State Find $\circumscribedellipsoid_{\convexset^{(k+1)}} = \ellipsoid(A^{(k+1)}, x^{(k+1)})$ of $\convexset^{(k+1)}$. 
		\State Set  $X = X \cup \lbrace x^{(k+1)} \rbrace$.
		\EndFor
		\State Find $x' = \min_{x \in X} \func(x)$ using $\oracle{C}$.
		\State \textbf{return} $x'$.
	\end{algorithmic}
\end{algorithm}

\begin{algorithm}[ht] 
	\caption*{\textbf{Function}  \textsc{PD-C}{$(x, \theta, A, t )$}}
	\begin{algorithmic}[1]
		\State Set $r = 1$, $\gamma = \pi/2$, $m=0$, $UD = \emptyset$, $p = e_{1}$.
		\While{$\gamma > \theta \wedge m < \dimension$}
		\State Set $\curlyset{d_{1}, \ldots, d_{m}} = UD$.
		\State Find $d_{i}$ such that $d_{m+1} = p$, $d_{i} \perp d_{j}$ for all $i\neq j \in [\dimension]$, and $\| d_{i} \| = 1$ for all $i \in [\dimension]$.
		\State  Set $\oracle{DP}(x, A^{-1/2}d_{i}) =$ \textsc{FDD-C}$(A, x_{0},  d, t)$ for all $i \in [\dimension]$.
		\If{$\exists i \in \curlyset{m+1, \ldots, \dimension}$, such that $\oracle{DP}(x, A^{-1/2}d_{i}) =unknown$}
		\State  Set $UD = UD \cup d_{i}$,  and $m= m+1$.
		\Else  
		\State Set $w_{i} =  d_{m+1} \oracle{DP}(x, A^{-1/2}d_{m+1}) \cos(\gamma) +  d_{i} \oracle{DP}(x, A^{-1/2}d_{i}) \sin(\gamma)$ for all $i \in [n]$.
		\State Set $p = \left( \sum_{i=m+1}^{\dimension} w_{i}/\dimension \right)/ \left\| \sum_{i=m+1}^{\dimension} w_{i}/\dimension  \right \|$, $\gamma = \cos^{-1}(\langle p, w_{m+2} \rangle)$, $r = \arcsin (\gamma)$. 
		\EndIf
		\EndWhile
		\State \textbf{if} $m\neq \dimension$ \textbf{then} \textbf{return} $p$, \textbf{else} \textbf{return} $e_{1}$.
	\end{algorithmic}
\end{algorithm}

\begin{algorithm}[ht] 
	\caption*{\textbf{Function} \textsc{FDD-C}{($A, x_{0},  d, t$)}}
	\begin{algorithmic}[1]
		\State Query $\oracle{\convexset}(x - tA^{-1/2}d, x), \oracle{\convexset}(x, x+ tA^{-1/2}d)$.
		\State \textbf{if} {$\func(x - tA^{-1/2}d) \leq \func(x) \wedge \func(x ) \leq \func(x+ tA^{-1/2}d)$} \textbf{then}  \textbf{return} $1$.
		\State \textbf{else if} {$\func(x - tA^{-1/2}d) < \func(x) \wedge \func(x ) < \func(x+ tA^{-1/2}d)$} \textbf{then} \textbf{return} $-1$. 
		\State \textbf{else} \textbf{return} $unknown$. 
	\end{algorithmic}
\end{algorithm}

\begin{theorem} \label{supp_theorem:C} 
Let $K =  \left \lceil 8 \dimension (\dimension + 1) \log\left( \frac{ \radiusofset{\convexset} \lipschitz}{\suboptimality}  \right)  \right \rceil.$
For an $\lipschitz$-Lipschitz, $\smoothness$-smooth, convex function $\func: \convexset \to \mathbb{R}$, Algorithm \ref{supp_algo:optimizeC} makes at most \[  2\dimension\left \lceil 2n\log(2\sqrt{2}\dimension) + \dimension \right \rceil K +  K
	\] queries to $\oracle{C}$ and the output $x'$ of Algorithm \ref{supp_algo:optimizeC} satisfies $\func(x') \leq \min_{x \in \convexset} \func(x) + \suboptimality$.
\end{theorem}

\begin{proof}[Proof of Theorem \ref{supp_theorem:C}]
	We first show that the output $x'$ of Algorithm \ref{supp_algo:optimizeC} satisfies $\func(x') \leq \min_{x \in \convexset} \func(x) + \suboptimality$. We then prove the bound on the number of queries.
	
	Note that all query points are in $\convexset$. By Lemma \ref{supp_lemma:feasibleinnerball} we know that every $\isotrans{A^{(k)}}{x^{(k)}}^{-1}(x)$ such that $\norm{x} \leq \frac{\sqrt{\eigenvalue_{\max}(A^{(k)})}}{2\dimension}$ are in $\convexset^{(k)}$ and consequently in $\convexset$. All queries have distance at most $\frac{\sqrt{\eigenvalue_{\max}(A^{(k)})}}{\kappa\dimension^{5/2}}$ from the origin in the isotropic coordinates. Since $\kappa \geq 1$, it implies that all query points are in $\convexset$.

	At iteration $k$, due to Lemma \ref{supp_lemma:ratioofellipsoidandcircle}, we have $\sqrt{\eigenvalue_{\max}(A^{(k)})} \leq \dimension \radiusofset{\convexset^{(k)}} \leq \dimension \radiusofset{\convexset}$. Consequently the radius $\radiusofset{\isotrans{A^{(k)}}{x^{(k)}}(\convexset^{(k)})}$ of $\convexset^{(k)}$ in isotropic coordinates is at most $\dimension \radiusofset{\convexset}$.
	
	Let $E_{k}$ denote the event that there does not exist a point $x^{*} \in \convexset^{(k)}$ such that $\func(x^{*}) = \min_{x \in \convexset} \func(x)$. Note that $E_{1}$ does not happen. Assume that $E_{1}, \ldots, E_{k}$ did not happen. We show that either event $E_{k+1}$ does not happen or the algorithm finds a near optimal point at iteration $k$. If $E_{1}, \ldots, E_{K+1}$ do not happen, then one of the ellipsoid centers are optimal as in the classical ellipsoid method.

	We consider $3$ cases:
	\begin{enumerate}
		\item $\norm{\nabla \left( \func \circ \isotransinverse{A^{(k)}}{x^{(k)}} \right) (0)} > \frac{\suboptimality}{\dimension \radiusofset{\convexset}}$ and $m = n$ when \textsc{PD-C} terminates,
		\item $\norm{\nabla \left( \func \circ \isotransinverse{A^{(k)}}{x^{(k)}} \right) (0)} > \frac{\suboptimality}{\dimension \radiusofset{\convexset}}$ and $m \neq n$  when \textsc{PD-C} terminates,
		\item $\norm{\nabla \left( \func \circ \isotransinverse{A^{(k)}}{x^{(k)}} \right) (0)} \leq \frac{\suboptimality}{\dimension \radiusofset{\convexset}}$.
	\end{enumerate}

	\textbf{Case 1:}
		We note that the function $\func \circ \isotransinverse{A^{(k)}}{x^{(k)}} $ is also $\smoothness$-smooth since we can only expand the coordinates via the isotropic transformation. In \textsc{PD-C}, if a unit vector $d$ is in $UD$, i.e., is unknown,  then due to $\smoothness$-smoothness we have $\absx{\innerproduct{ \nabla \left( \func \circ \isotransinverse{A^{(k)}}{x^{(k)}} \right) (0)}{d}} \leq \frac{\min(\suboptimality, \sqrt{\eigenvalue_{\max}(A^{k})} )}{\kappa \dimension^{5/2} \max(\radiusofset{\convexset},1)}$. 
		
	 Since $m =n$, i.e., all basis directions are unknown, we have $\absx{\innerproduct{ \nabla \left( \func \circ \isotransinverse{A^{(k)}}{x^{(k)}} \right)(0)}{d_{i}}} \leq \frac{\min(\suboptimality, \sqrt{\eigenvalue_{\max}(A^{k})} )}{\kappa \dimension^{5/2} \max(\radiusofset{\convexset},1)}$ for all $i\in [\dimension]$. Since $d_{i} \perp d_{j}$ for all $i \neq j \in [\dimension]$, and $\norm{d_{i}} = 1$ for all $i \in [\dimension]$ by construction, we have $\norm{\nabla \left( \func \circ \isotransinverse{A^{(k)}}{x^{(k)}} \right)(0)} \leq\frac{\min(\suboptimality, \sqrt{\eigenvalue_{\max}(A^{k})} )}{\kappa \dimension^{2} \max(\radiusofset{\convexset},1)}.$ This implies that $\func(x^{(k)}) \leq \func(x^{*}) + \frac{\radiusofset{\convexset}\min(\suboptimality, \sqrt{\eigenvalue_{\max}(A^{k})} )}{\kappa \dimension \max(\radiusofset{\convexset},1)} \leq \func(x^{*}) + \suboptimality$ since $\kappa \geq 1$, the function in isotropic coordinates is convex, , $\radiusofset{\isotrans{A^{(k)}}{x^{(k)}}(\convexset^{(k)})} \leq \dimension \radiusofset{\convexset}$, and there exists a minimizer $x^{*} \in \convexset^{(k)}$.
	
		\textbf{Case 2:} Let $p^{j}$ denote the value of $p$ at iteration $j$ of \textsc{PD-C}. Assume that $\angle\left(p^{j}, Proj_{UD^{\bot}}\left(\nabla \left( \func \circ \isotransinverse{A^{(k)}}{x^{(k)}} \right)(0)\right) \right) \leq \gamma$ while \textsc{PD-C} runs.  If a new unknown direction $d_{i}$ is detected at iteration $j$ of \textsc{PD-C}, then we have $\angle\left(p^{j+1}, Proj_{(UD \cup \curlyset{d_{i}})^{\bot}}\left(\nabla \left( \func \circ \isotransinverse{A^{(k)}}{x^{(k)}} \right)(0)\right) \right) \leq \gamma$  since $p^{j+1}$ and $d_{i}$ are orthogonal. Therefore, the angle $\angle\left(p, Proj_{UD^{\bot}}\left(\nabla \left( \func \circ \isotransinverse{A^{(k)}}{x^{(k)}} \right)(0)\right) \right)$ does not increase when a new unknown direction is detected. If there is no new unknown direction, then $Proj_{UD^{\bot}}\left(\nabla \left( \func \circ \isotransinverse{A^{(k)}}{x^{(k)}} \right)(0)\right)$ is in a hyperoctant in the subspace defined by $Span(UD^{\bot}).$ We have 
	\[	\angle\left(p^{j+1}, Proj_{UD^{\bot}}\left(\nabla \left( \func \circ \isotransinverse{A^{(k)}}{x^{(k)}} \right)(0)\right) \right)  \leq \arccos(\innerproduct{p^{j+1}}{w_{m+2}}) \leq \arcsin\left( \sqrt{\frac{\dimension-1}{\dimension}} \sin(\gamma)  \right) \] by Lemma \ref{supp_lemma:smallergradientcones}. Since the angle $\angle\left(p, Proj_{UD^{\bot}}\left(\nabla \left( \func \circ \isotransinverse{A^{(k)}}{x^{(k)}} \right)(0)\right) \right)$ decreases by a constant factor if there is no new unknown directions and \textsc{PD-C} can detect an unknown direction at most $\dimension -1$ times, we have $\angle\left(p, Proj_{UD^{\bot}}\left(\nabla \left( \func \circ \isotransinverse{A^{(k)}}{x^{(k)}} \right)(0)\right) \right) \leq \arcsin(1/(2\sqrt{2}\dimension))$ when \textsc{PD-C} terminates. This implies \begin{equation} \label{supp_eq:projontoAD}
	\frac{\innerproduct{p}{Proj_{UD^{\bot}}\left(\nabla \left( \func \circ \isotransinverse{A^{(k)}}{x^{(k)}} \right)(0)\right)}}{\norm{Proj_{UD^{\bot}}\left(\nabla \left( \func \circ \isotransinverse{A^{(k)}}{x^{(k)}} \right)(0)\right)}} \geq \sqrt{1 - \frac{1}{8\dimension^2}}.
	\end{equation} Since $p \not\in Span(UD)$, \eqref{supp_eq:projontoAD} implies that \begin{equation} \label{supp_eq:projontoAD2}
	\frac{\innerproduct{p}{ \nabla \left( \func \circ \isotransinverse{A^{(k)}}{x^{(k)}} \right)(0)}}{\norm{Proj_{UD^{\bot}}\left(\nabla \left( \func \circ \isotransinverse{A^{(k)}}{x^{(k)}} \right)(0)\right)}} \geq \sqrt{1 - \frac{1}{8\dimension^2}}.
	\end{equation}

	When \textsc{PD-C} terminates, we have \[ {\norm{Proj_{UD}\left(\nabla \left( \func \circ \isotransinverse{A^{(k)}}{x^{(k)}} \right)(0)\right)}} \leq \frac{\sqrt{m}\min(\suboptimality, 1)}{\kappa \dimension^{2} \max(\radiusofset{\convexset},1)} \leq \frac{\sqrt{\dimension}\suboptimality}{\kappa\dimension^{5/2} \radiusofset{\convexset}} = \frac{\suboptimality}{\kappa\dimension^{2} \radiusofset{\convexset}} .  \]  Using this we get  
	\begin{subequations}
		\begin{align*}
		\norm{\nabla \left( \func \circ \isotransinverse{A^{(k)}}{x^{(k)}} \right)(0)} &= \norm{Proj_{UD}\left(\nabla \left( \func \circ \isotransinverse{A^{(k)}}{x^{(k)}} \right)(0)\right) + Proj_{UD^{\bot}}\left(\nabla \left( \func \circ \isotransinverse{A^{(k)}}{x^{(k)}} \right)(0)\right)} 
		\\
		&\leq \norm{Proj_{UD^{\bot}}\left(\nabla \left( \func \circ \isotransinverse{A^{(k)}}{x^{(k)}} \right)(0)\right)}   + \frac{\suboptimality}{\kappa\dimension^{2} \radiusofset{\convexset}}.
		\end{align*}
	\end{subequations}

	Since $\norm{\nabla \left( \func \circ \isotransinverse{A^{(k)}}{x^{(k)}} \right)(0)} > \frac{\suboptimality}{\dimension \radiusofset{\convexset}}$, we have 
	\begin{equation} \label{supp_eq:projontoADneg}
	\frac{\norm{Proj_{UD^{\bot}}\left(\nabla \left( \func \circ \isotransinverse{A^{(k)}}{x^{(k)}} \right)(0)\right)} }{\norm{\nabla \left( \func \circ \isotransinverse{A^{(k)}}{x^{(k)}} \right)(0)}}  > 1 - \frac{1}{\kappa \dimension} = \sqrt{1- \frac{1}{8 \dimension^2}}.
	\end{equation}
	
	By combining \eqref{supp_eq:projontoAD2} and \eqref{supp_eq:projontoADneg}, we finally get \begin{subequations}
		\begin{align}
		&\frac{\innerproduct{p}{ \nabla \left( \func \circ \isotransinverse{A^{(k)}}{x^{(k)}} \right)(0)}}{\norm{ \nabla \left( \func \circ \isotransinverse{A^{(k)}}{x^{(k)}} \right)(0)}} 
		\\ &= \frac{\innerproduct{p}{ \nabla \left( \func \circ \isotransinverse{A^{(k)}}{x^{(k)}} \right)(0)}}{\norm{Proj_{UD^{\bot}}\left(\nabla \left( \func \circ \isotransinverse{A^{(k)}}{x^{(k)}} \right)(0)\right)}} 		\frac{\norm{Proj_{UD^{\bot}}\left(\nabla \left( \func \circ \isotransinverse{A^{(k)}}{x^{(k)}} \right)(0)\right)}}{\norm{ \nabla \left( \func \circ \isotransinverse{A^{(k)}}{x^{(k)}} \right)(0)}} 
		\\
		&\geq \sqrt{1- \frac{1}{8 \dimension^2}}\sqrt{1- \frac{1}{8 \dimension^2}}
		\\
		&\geq \sqrt{1- \frac{1}{4 \dimension^2}}. \label{supp_ineq:lowdegree}
		\end{align}
	\end{subequations}
	
	We note that \eqref{supp_ineq:lowdegree} implies that when $m\neq \dimension$ and $\norm{ \nabla \left( \func \circ \isotransinverse{A^{(k)}}{x^{(k)}} \right)(0)}  > \frac{\suboptimality}{\dimension \radiusofset{\convexset}}$, we have $ \nabla \left( \func \circ \isotransinverse{A^{(k)}}{x^{(k)}} \right)(0) \in \cone(p, \arcsin(1/(2\dimension)))$. This implies that the gradient estimate is accurate and the ellipsoid cut only removes ascent directions. If the gradient pruning algorithm succeeds in the $k$th iteration then by Lemma \ref{supp_lemma:ellipsoid}, $C^{(k)} \setminus C^{(k+1)}$ only includes the ascent points.  After every iteration $k$, there exists a $x^{*} \in \convexset^{(k)}$ such that $\func(x^{*}) = \min_{x \in \convexset} \func(x)$.  Therefore, event $E_{k+1}$ does not happen.
	
	\textbf{Case 3:} We have  $\func(x^{(k)}) \leq \func(x^{*}) + \suboptimality$ since the function in isotropic coordinates is convex, $\radiusofset{\isotrans{A^{(k)}}{x^{(k)}}(\convexset^{(k)})} \leq \dimension \radiusofset{\convexset}$, and there exists a minimizer $x^{*} \in \convexset^{(k)}$.

	If Case $1$ or $3$ happens, the output point $x'$ of Algorithm \ref{supp_algo:optimizeC} satisfies $\func(x') \leq \func(x^{(k)}) \leq \func(x^{*}) + \suboptimality$ since Algorithm \ref{supp_algo:optimizeC} compares the ellipsoid centers before termination.  If Case $1$ or $3$ does not happen, then event $E_{1}, \ldots, E_{K+1}$ does not happen, i.e., the ellipsoid method proceeds successfully. Without loss of generality we assume that Case $1$ or $3$ does not happen.
	
	Since $\func$ is $\lipschitz$-Lipschitz the volume of the set $\curlyset{x | x \in \convexset, \func(x) \leq \func(x^{*}) + \suboptimality}$ is at least $\volumeofunitball \left( \frac{\suboptimality}{ \lipschitz}\right)^{\dimension}$. Let $K = \left \lceil 8\dimension (\dimension+1) \log\left( \frac{ \radiusofset{\convexset} \lipschitz}{\suboptimality} \right)  \right \rceil$. Due to Lemma \ref{supp_lemma:ellipsoid}, we have $Vol\left(\circumscribedellipsoid_{\convexset^{(K)}} \right)< \volumeofunitball \left( \frac{\suboptimality}{ \lipschitz}\right)^{\dimension}$. Therefore, there exists a point $x$ such that $x \not \in \convexset^{(K)}$ and $\func(x) \leq \func(x^{*}) + \suboptimality$. 
	
	Since the function value of every discarded point in $\convexset \setminus \convexset^{(K)}$ is greater than or equal to $\func(x^{k})$ for some $1 \leq k \leq K$, we have $\func(x^{k}) \leq \func(x^{*}) + \suboptimality$ for some $1 \leq k \leq K$. Therefore, the output point $x'$ satisfies $\func(x') \leq \min_{x \in X} \func(x) + \suboptimality\leq \func(x^{*}) + \suboptimality.$ 
	
	We now prove the bound on the number of queries. The gradient pruning algorithm starts with $\gamma = \pi/2$. As shown in Lemma \ref{supp_lemma:smallergradientcones}, if the there is no new unknown direction, each iteration satisfies $\frac{\sin(\gamma')}{\sin(\gamma)} \leq \sqrt{ \frac{\dimension - |UD| -1}{\dimension - |UD|}} -  \sqrt{ \frac{\dimension -1}{\dimension }}$ where $\gamma'$ is the new value assigned to $\gamma$. After $k$ iterations we have $\sin(\gamma) \leq \sqrt{ \frac{\dimension-1}{\dimension}}^{k} \leq e^{\frac{-k}{2n}}$. Since $\theta = \arcsin(1/(2\sqrt{2}\dimension))$, the gradient pruning algorithm stops after at most $\left \lceil 2n\log(2\sqrt{2}\dimension) \right \rceil$ iterations where we make at most $2\dimension$ queries in each iteration. Note that we can detect at most $\dimension$ unknown directions while running the gradient pruning algorithm. Therefore, the gradient pruning algorithm makes at most $2\dimension\left \lceil 2n\log(2\sqrt{2}\dimension) + \dimension \right \rceil$ queries to the oracle.
	
	The for loop in Algorithm \ref{supp_algo:optimizeC} has $\left \lceil 8\dimension(\dimension + 1 ) \log\left( \frac{ \radiusofset{\convexset} \lipschitz}{\suboptimality}  \right)   \right \rceil$ iterations. Therefore, the total number of queries is at most \[2\dimension\left \lceil 2n\log(2\sqrt{2}\dimension) + \dimension \right \rceil  \left \lceil 8 \dimension (\dimension + 1) \log\left( \frac{ \radiusofset{\convexset} \lipschitz}{\suboptimality}  \right)   \right \rceil \] before the last comparison step. The set $X$ has at most $\left \lceil 8 \dimension (\dimension + 1) \log\left( \frac{ \radiusofset{\convexset} \lipschitz}{\suboptimality}  \right) + 1  \right \rceil$ elements. Finding the smallest function value  requires $\left \lceil 8 \dimension (\dimension + 1) \log\left( \frac{2 \radiusofset{\convexset} \lipschitz}{\suboptimality}  \right)  \right \rceil$ queries to the comparator oracle. Thus, the total number of queries is at most \[ 2\dimension\left \lceil 2n\log(2\sqrt{2}\dimension) + \dimension \right \rceil  \left \lceil 8 \dimension (\dimension + 1) \log\left( \frac{ \radiusofset{\convexset} \lipschitz}{\suboptimality}  \right)   \right \rceil +  \left \lceil 8 \dimension (\dimension + 1) \log\left( \frac{ \radiusofset{\convexset} \lipschitz}{\suboptimality}  \right)  \right \rceil.
	\]
\end{proof}

\subsection{Proof of Theorem \ref{supp_theorem:regret}}
\setcounter{algorithm}{2}

\begin{algorithm}[ht] 
	\caption{The low regret algorithm \textsc{Regret-NV($\timehorizon, \failureprob$)} for the noisy value oracle } \label{supp_algo:regretNV}
	\begin{algorithmic}[1]
		\State  Set $\convexset^{(1)} = \convexset$. Find $\circumscribedellipsoid_{\convexset^{(1)}} = \ellipsoid(A^{(1)}, x^{(1)})$.  Set $X = \lbrace x^{(1)} \rbrace$.
		\State Set $K = \left \lceil 8\dimension(\dimension + 1) \log\left(  2 \radiusofset{\convexset} \lipschitz \timehorizon^{0.25} \right)  \right \rceil$, $\tau = \ceilx{ 32 \standarddev^2  \dimension^4 \log\left( \frac{2}{\failureprob'} \right)}$, $\failureprob' = \frac{\failureprob}{4 \dimension K \log_{16}\left( \frac{15 \timehorizon}{2 \dimension} \right)}$.
		\For{$ k =1, \ldots, K$} \Comment{Phase 1}
		\State Set $d = \frac{\min\left(\sqrt{\eigenvalue_{\max}(A^{(k)})}, 1\right)}{2 \dimension  }$.
		\State Set $\Delta = \frac{d  \left( 2+ \smoothness\eigenvalue_{\max}(A^{(k)}) \right)}{2\eigenvalue_{\max}(A^{(k)})}$.
		\For{$i = 0, 1, \ldots$} \Comment{Case 2}
		\State Set $d_{i} = d/2^{i}$, $\Delta_{i} = \Delta/2^{i}$, $\tau_{i} = 2^{4i} \tau$
		\State Query $\tau_{i}$ times $\oracle{NV}(x^{(k)})$ and $\oracle{NV}(\isotransinverse{A^{(k)}}{ x^{(k)}}(de_{j}))$  for all $i \in [\dimension]$. 
		\State For every query point $x$, set $\hat{\oraclewithoutname}^{NV}(x)$ as the mean of queries for point $x$.
		\State Estimate the gradient $p$ using the mean values $\hat{\oraclewithoutname}^{NV}(x)$.
		\If{$\left( \| p\| > \sqrt{\dimension} \Delta_{i} \right) \wedge \left(\arcsin\left(\frac{\sqrt{\dimension} \Delta_{i}}{\|p\|}\right) \leq \arcsin\left(\frac{1}{2\dimension}\right)\right)$} \Comment{Case 1}
		\State Set $\convexset^{(k+1)} = \convexset^{(k)} \cap  \ellipsoid(A^{(k)}, x^{(k)}) \cap \isotransinverse{A}{x^{(k)}} \left(\curlyset{ x | \innerproduct{p / \norm{p}}{ x } \leq 1/(2\dimension) }\right)$.
		\State Find $\circumscribedellipsoid_{\convexset^{(k+1)}} = \ellipsoid(A^{(k+1)}, x^{(k+1)})$. 
		\State Set $X = X \cup \lbrace x^{(k+1)} \rbrace$.
		\State \textbf{break}
		\EndIf
		\EndFor
		\EndFor
		\State For all $x\in X$, query $\oracle{NV}(x)$, $\ceilx{32 \standarddev^2 \sqrt{\timehorizon} \log\left( \frac{2(K+1)}{\failureprob}\right) }$ times. Set $x'$ to the point with the highest empirical mean.  \Comment{Phase 2}
		\State Repeatedly query $\oracle{NV}(x')$. \Comment{Phase 3}
	\end{algorithmic}
	
\end{algorithm}

\begin{theorem} \label{supp_theorem:regret}
 Let $K = \left \lceil 8\dimension(\dimension + 1) \log\left(  2 \radiusofset{\convexset} \lipschitz \timehorizon^{0.25} \right)  \right \rceil$, $\failureprob' = \failureprob/ \left( 4 \dimension K \log_{16}\left( \frac{15 \timehorizon}{2 \dimension} \right) \right)$, and $\tau = \ceilx{ 8\standarddev^2 \smoothness^2 \dimension^4 \log\left( \frac{2}{\failureprob'} \right)}$. For an $\lipschitz$-Lipschitz, $\smoothness$-smooth, convex function $\func: \convexset \to \mathbb{R}$, a given failure probability  $\failureprob >0$, and a time horizon $\timehorizon$, Algorithm \ref{supp_algo:regretNV} has a regret of at most $ K\left(\radiusofset{\convexset} \lipschitz \tau   + 5  \timehorizon^{0.75}  \dimension^{-0.25}\max\left( \dimension \radiusofset{\convexset}, 1 \right)(1+\smoothness)   \tau^{0.25} \right) + 
 (K+1) \ceilx{32 \standarddev^2 \sqrt{\timehorizon} \log\left( \frac{2(K+1)}{\failureprob}\right) } \radiusofset{\convexset} \lipschitz + \timehorizon^{0.75}$
 with probability at least $1-\failureprob$.
\end{theorem}

\begin{proof}[Proof of Theorem \ref{supp_theorem:regret}]
	We first show that all queries are feasible. We note that during Phase $1$, all queries have distance at most $ \sqrt{\eigenvalue_{\max}(A^{(k)})}/(2 \dimension)$ from the origin in the isotropic coordinates where $\sqrt{\eigenvalue_{\max}(A^{(k)})}$ is the radius of the current convex set in the isotropic coordinates. By Lemma \ref{supp_lemma:feasibleinnerball}, all queries in Phase $1$ are feasible. The query points in Phases $2$ and $3$ are ellipsoid centers which are feasible due to Lemma \ref{supp_lemma:feasibleinnerball}.
	
	We analyze the regret induced by the inner for loop. Let $D$ be the current convex set such that $\circumscribedellipsoid_{D} = \ellipsoid(A, x)$.
	We first show that the gradient estimate estimation is accurate with high probability in the isotropic coordinates. Since the isotropic transformation can only stretch the coordinates, the function in the isotropic coordinates is also $\smoothness$-smooth and $\lipschitz$-Lipschitz. We have  
\begin{equation} \label{supp_eq:gradientest2}
\left\lvert \frac{\left( \func \circ \isotransinverse{A}{x} \right) (0) - \left( \func \circ \isotransinverse{A}{x} \right) (d_{i} e_{j})}{d_{i}}  -  \frac{\innerproduct{\nabla \left( \func \circ \isotransinverse{A}{x} \right) (0)}{d_{i} e_{j}}}{d_{i}}   \right\lvert \leq \frac{\smoothness d_{i}}{2}  
\end{equation}
	due to $\smoothness$-smoothness.

	At the $i$-th iteration of the inner for loop, 	the directional derivative estimate in direction $e_{j}$ is  $p_{j} = \frac{\hat{\oraclewithoutname}^{NV}(\isotransinverse{A}{ x}(d_{i}e_{j})) -  \hat{\oraclewithoutname}^{NV}(x)}{d_{i}}$. We have 
	\begin{subequations}
		\begin{align}
		&\Pr\left(\left \lvert \frac{\hat{\oraclewithoutname}^{NV}(\isotransinverse{A}{ x}(d_{i}e_{j})) -  \hat{\oraclewithoutname}^{NV}(x)}{d_{i}} - \frac{\left( \func \circ \isotransinverse{A}{x} \right) (0) - \left( \func \circ \isotransinverse{A}{x} \right) (d_{i} e_{j})}{d_{i}} \right \rvert > \frac{d_{i}}{\min \left(\eigenvalue_{\max}(A), 1 \right)} \right) 
		\\
		& \quad \leq 2 \exp\left( - \frac{d_{i}^{4}  \tau_{i}}{2 \standarddev^2  \min \left(\eigenvalue_{\max}(A), 1 \right)^2 }\right) \leq \delta'
		\end{align}
	\end{subequations}
	for each direction $e_{j}$. Using this in \eqref{supp_eq:gradientest2}, the directional derivative estimate satisfies $\left \lvert p_{j}  - \innerproduct{\nabla \left( \func \circ \isotransinverse{A}{x} \right) (0)}{e_{j}}\right \rvert \leq \Delta_{i}$ at point $0$ with probability at least $1-\delta'$. Consequently, we have $\norm{\nabla \left( \func \circ \isotransinverse{A}{x} \right) (0) - p} \leq \sqrt{\dimension} \Delta_{i} $ with probability at least $1-2n\delta'.$ 
	
	If Case $2$ happens, then we have $ \norm{\nabla \left( \func \circ \isotransinverse{A}{x} \right) (0) } < (2\dimension + 1)\sqrt{\dimension} \Delta_{i}$ since $ \norm{p} < 2\dimension\sqrt{\dimension} \smoothness \Delta_{i}$. Since $\isotrans{A}{x}\func$ is $\smoothness$-smooth, the norm of the gradient is smaller than $ (2\dimension + 1)\sqrt{\dimension}  \Delta_{i} + \smoothness d_{i}$ for every query point.
	
	If Case $1$ happens then $\nabla \left( \func \circ \isotransinverse{A}{x} \right) (0) \in \cone(p, \arcsin(1/(2\dimension)))$, and the ellipsoid algorithm proceeds successfully. Note that the elliposid cuts happen only when Case $1$ happens. Since every discarded point $y$ satisfies $\func(y) \geq \func(x)$, the set $D$ always contains a minimizer $x^{*}$.

	We first show that Case $2$ can happen at most $\log_{16}\left( \frac{15 \timehorizon}{2 \dimension \tau } \right)$ times. In the $i$th iteration of the inner for loop, we make $2^{4i} \tau$ queries for two points in every dimension. Let $W$ be the number of iterations of the inner for loop.  We have 
	\begin{subequations}
		\begin{align}
		\timehorizon &= \sum_{i=0}^{W} 2^{4i} 2 \dimension \tau
		\\
		&=  \frac{(16^{W+1} - 1) 2\dimension\tau}{15}.
		\end{align}
	\end{subequations}
	By rearranging the terms, we get $W = \log_{16}\frac{15 \timehorizon}{2\dimension \tau} - 1$. Therefore, the maximum value of $i$ is $\log_{16}\left( \frac{15 \timehorizon}{2 \dimension \tau } \right)$.
	
	For each iteration of inner for loop the probability of failure is less than or equal to $2\dimension\failureprob'$. Since the maximum value of $i$ is $\log_{16}\left( \frac{15 \timehorizon}{2 \dimension \tau } \right)$, the total probability of failure is less than or equal to $2\dimension\failureprob'\log_{16}\left( \frac{15 \timehorizon}{2 \dimension \tau } \right)$.
	
	We now bound the regret for each iteration of the inner loop assuming that the gradient estimation did not fail. If $i=0$, then the regret of each query is $\radiusofset{D} \lipschitz$ since there exists a minimizer $x^{*} \in D$. If $i>0$, then $\norm{\nabla \left( \func \circ \isotransinverse{A}{x} \right) (0)} < 2(2\dimension + 1)\sqrt{\dimension}\Delta_{i}$ since Case $1$ did not happen in iteration $i-1$. Due to $\smoothness$-smoothness, the norm of gradient at the query points is smaller than $2(2\dimension + 1)\sqrt{\dimension}\Delta_{i} \Delta_{i} + \smoothness d_{i}$ in isotropic coordinates. The regret of each query is smaller than $ \sqrt{\eigenvalue_{\max}(A)} \dimension (4\dimension\sqrt{\dimension} \Delta_{i} + \smoothness d_{i})$ since $D$ contains a minimizer $x^{*}$ and the radius of $D$ is $\sqrt{\eigenvalue_{\max}(A)}$ in the isotropic coordinates. The total regret induced by the inner for loop is less than or equal to
	\begin{subequations}
		\begin{align}
		&\radiusofset{D} \lipschitz \tau + \sum_{i=1}^{\log_{16}\left( \frac{15 \timehorizon}{2 \dimension \tau } \right)} \sqrt{\eigenvalue_{\max}(A)}  (2(2\dimension + 1)\sqrt{\dimension} \Delta_{i} + \smoothness d_{i}) \tau_{i}
		\\
		& \quad = \radiusofset{D} \lipschitz \tau + \sum_{i=1}^{\log_{16}\left( \frac{15 \timehorizon}{2 \dimension \tau } \right)} 2^{3i}\sqrt{\eigenvalue_{\max}(A)} (2(2\dimension + 1)\dimension\sqrt{\dimension} \Delta + \smoothness d)  \tau
		\\
		&\quad = \radiusofset{D} \lipschitz \tau + \frac{8^{\log_{16}\left( \frac{15 \timehorizon}{2 \dimension \tau } \right)}  - 8}{7}\sqrt{\eigenvalue_{\max}(A)}(2(2\dimension + 1)\sqrt{\dimension} \Delta + \smoothness d) \tau
		\\
		&\quad = \radiusofset{D} \lipschitz \tau  + \frac{15^{3/4} \timehorizon^{3/4}}{7 (2^{3/4} \dimension^{3/4} \tau^{3/4})}\sqrt{\eigenvalue_{\max}(A)}2(2\dimension + 1)\sqrt{\dimension} \Delta + \smoothness d) \tau
		\\
		&\quad \leq \radiusofset{D} \lipschitz \tau  + \timehorizon^{3/4}\sqrt{\eigenvalue_{\max}(A)} \dimension^{-3/4} (2(2\dimension + 1)\sqrt{\dimension} \Delta + \smoothness d) \tau^{1/4} \label{supp_ineq:15power}
		\\
		&\quad = \radiusofset{D} \lipschitz \tau  + \timehorizon^{3/4} \sqrt{\eigenvalue_{\max}(A)} \dimension^{-3/4} 
		\\
		& \quad \quad \left(2\left(2\dimension + 1\right)\sqrt{\dimension} \frac{2+ \smoothness \min\left(\eigenvalue_{\max}(A),1 \right)}{2\dimension \sqrt{ \min\left(\eigenvalue_{\max}(A),1 \right)}} +  \frac{ \smoothness \sqrt{ \min\left(\eigenvalue_{\max}(A),1 \right)}}{\dimension} \right)  \tau^{1/4}
		\\
		&\quad \leq \radiusofset{D} \lipschitz \tau  + \timehorizon^{3/4} \sqrt{\eigenvalue_{\max}(A)} \dimension^{-3/4} \left(2\left(2\dimension + 1\right)\sqrt{\dimension} \frac{1+ \smoothness \min\left(\eigenvalue_{\max}(A),1 \right)}{\dimension \sqrt{ \min\left(\eigenvalue_{\max}(A),1 \right)}} \right)  \tau^{1/4} \label{supp_ineq:2nplus1sqrtn} 
		\\
		&\quad \leq \radiusofset{D} \lipschitz \tau  + 5\timehorizon^{3/4} \sqrt{\eigenvalue_{\max}(A)} \dimension^{-1/4} \left( \frac{1+ \smoothness \min\left(\eigenvalue_{\max}(A),1 \right)}{\sqrt{ \min\left(\eigenvalue_{\max}(A),1 \right)}} \right)  \tau^{1/4}  \label{supp_ineq:2times2nplus1lessthan6n} 
		\\
		&\quad \leq \radiusofset{D} \lipschitz \tau   + 5  \timehorizon^{3/4}  \dimension^{-1/4}\max\left( \sqrt{\eigenvalue_{\max}(A)}, 1 \right)(1+\smoothness)   \tau^{1/4} 
		\\
		&\quad \leq \radiusofset{D} \lipschitz \tau   + 5  \timehorizon^{3/4}  \dimension^{-1/4}\max\left( \dimension \radiusofset{D}, 1 \right)(1+\smoothness)   \tau^{1/4}
		\\
		&\quad \leq \radiusofset{\convexset} \lipschitz \tau   + 5  \timehorizon^{3/4}  \dimension^{-1/4}\max\left( \dimension \radiusofset{\convexset}, 1 \right)(1+\smoothness)   \tau^{1/4} \label{supp_ineq:radiusDsmallerthanradiusC}
		\end{align}
	\end{subequations}
	where \eqref{supp_ineq:15power} is due to $(15/2)^{3/4}/7 \leq 1$,  \eqref{supp_ineq:2nplus1sqrtn} is due to $(2\dimension+1)\sqrt{\dimension} +1 \leq (2\dimension+1)\sqrt{\dimension}$, and \eqref{supp_ineq:2times2nplus1lessthan6n} is due to $2(2\dimension + 1) \leq 5\dimension$. Inequality \eqref{supp_ineq:radiusDsmallerthanradiusC} follows from $\eigenvalue_{\max}(A) \leq \dimension\radiusofset{D}$  for the convex set $D$ by Lemma \ref{supp_lemma:ratioofellipsoidandcircle} and $\radiusofset{D} \leq \radiusofset{C}$.

	Since the outer for loop repeats at most $K$ times, the total regret incurred during Phase $1$ is at most $K\left(\radiusofset{\convexset} \lipschitz \tau   + 5  \timehorizon^{3/4}  \dimension^{-1/4}\max\left( \dimension \radiusofset{\convexset}, 1 \right)(1+\smoothness)   \tau^{1/4} \right)$ with probability at least $1 - 2\dimension K \failureprob'\log_{16}\left( \frac{15 \timehorizon}{2 \dimension \tau } \right) = 1- \failureprob \log_{16}\left( \frac{15 \timehorizon}{2 \dimension \tau } \right) / \left( 2 \log_{16}\left( \frac{15 \timehorizon}{2 \dimension} \right)\right) $. Since $\tau \geq 1$, we have the probability of failure is less than $1-\failureprob/2$.
	
	If Case $1$ happens in the $k$th iteration of the outer loop, then $C^{(k)} \setminus C^{(k+1)}$ only includes the ascent points by Lemma \ref{supp_lemma:ellipsoid}.  Since $\func$ is $\lipschitz$-Lipschitz the volume of the set $\curlyset{x | x \in \convexset, \func(x) \leq \func(x^{*}) + \suboptimality}$ is at least $\volumeofunitball \left( \frac{\suboptimality}{ \lipschitz}\right)^{\dimension}$. If the iteration $K$ happens, then due to Lemma \ref{supp_lemma:ellipsoid}, we have $Vol\left(\circumscribedellipsoid_{\convexset^{(K)}} \right)< \volumeofunitball \left( \frac{\suboptimality}{ \lipschitz}\right)^{\dimension}$. Therefore, there exists a point $x$ such that $x \not \in \convexset^{(K)}$ and $\func(x) \leq \func(x^{*}) + \timehorizon^{-1/4}/2$. 
	
	Since the function value of every discarded point in $\convexset \setminus \convexset^{(K)}$ is greater than or equal to $\func(x^{k})$ for some $1 \leq k \leq K$, we have $\func(x^{k}) \leq \func(x^{*}) + \timehorizon^{-1/4}/2$ for some $1 \leq k \leq K$.
	
	By the Hoeffding's inequality, the point $x'$ with the highest empirical mean satisfies $\func(x^{k}) \leq \func(x^{*}) + \timehorizon^{-1/4}$ with probability at least $1-\delta/2$. 
	
	Since set $X$ has $K+1$ elements, the regret incurred during Phase $2$ is at most
	\begin{subequations}
		\begin{align}
		(K+1) \ceilx{32 \standarddev^2 \sqrt{\timehorizon} \log\left( \frac{2(K+1)}{\failureprob}\right) } \radiusofset{\convexset} \lipschitz.
		\end{align}
	\end{subequations}  
	
	Since the output point $x'$ satisfies $\func(x^{k}) \leq \func(x^{*}) + \timehorizon^{-1/4}$. Therefore, the regret incurred during Phase $3$ is at most $\timehorizon^{3/4}$. 
	
	Therefore, the regret of Algorithm \ref{supp_algo:regretNV} is at most $
	K\left(\radiusofset{\convexset} \lipschitz \tau   + 5  \timehorizon^{3/4}  \dimension^{-1/4}\max\left( \dimension \radiusofset{\convexset}, 1 \right)(1+\smoothness)   \tau^{1/4} \right) + 
	(K+1) \ceilx{32 \standarddev^2 \sqrt{\timehorizon} \log\left( \frac{2(K+1)}{\failureprob}\right) } \radiusofset{\convexset} \lipschitz + \timehorizon^{3/4}
	$
	with probability at least $1-\failureprob$.
	
\end{proof}